\documentclass[reqno]{amsart}
\usepackage{amsmath,amsfonts,amssymb,amsthm}
\usepackage[usenames,dvipsnames]{xcolor}
\usepackage[colorlinks=true]{hyperref} 
\usepackage{url}
\hypersetup{linkcolor=Violet,citecolor=PineGreen}
\newtheorem{teor}{Theorem}[section]
\newtheorem{lema}[teor]{Lemma}

\theoremstyle{definition}
\newtheorem{defi}[teor]{Definition}
\newtheorem{exa}[teor]{Example}

\newtheorem{hipos}[teor]{Hypotheses}

\newtheorem{nota}[teor]{Remark}
\newtheorem{notas}[teor]{Remarks}
\numberwithin{equation}{section}

\newcommand{\R}{\mathbb R}

\newcommand{\C}{\mathbb{C}}

\newcommand{\M}{\mathbb{M}}
\newcommand{\s}{\mathbb{S}}

\newcommand{\mP}{\mathcal{P}}

\newcommand{\ep}{\varepsilon}

\newcommand{\mI}{\mathcal{I}}

\newcommand{\mQ}{\mathcal{Q}}

\newcommand{\W}{\Omega}
\newcommand{\w}{\omega}

\newcommand{\bv}{\mathbf v}
\newcommand{\bc}{\mathbf c}
\newcommand{\bd}{\mathbf d}
\newcommand{\bw}{\mathbf w}
\newcommand{\bx}{\mathbf x}
\newcommand{\by}{\mathbf y}
\newcommand{\bz}{\mathbf z}
\newcommand{\bu}{\mathbf u}

\newcommand{\bcero}{\mathbf 0}
\newcommand{\wit}{\widetilde}
\newcommand{\ws}{\w{\cdot}s}
\newcommand{\wt}{\w{\cdot}t}

\newcommand{\n}[1]{\left\|#1\right\|}
\newcommand{\Frac}[2]{\displaystyle\frac{#1}{#2}}

\newcommand{\lsm}{\left[\begin{smallmatrix}}
\newcommand{\rsm}{\end{smallmatrix}\right]}

\DeclareMathOperator{\Supp}{Supp}

\begin{document}
\title[On the solvability of the Yakubovich minimization problem]
{On the solvability of the Yakubovich linear-quadratic infinite horizon
minimization problem}
\author[R. Fabbri]{Roberta Fabbri}
\author[C. N\'{u}\~{n}ez]{Carmen N\'{u}\~{n}ez}
\address{Dipartimento di Matematica e Informatica \lq Ulisse Dini\rq,
Universit\`{a} degli Studi di Firenze, Via Santa Marta 3, 50139 Firenze, Italy.}
\address{Departamento de Matem\'{a}tica Aplicada,
Universidad de Valladolid, Paseo del Cauce 59, 47011 Valladolid, Spain.}
\email[Roberta Fabbri]{roberta.fabbri@unifi.it}
\email[Carmen N\'{u}\~{n}ez]{carnun@wmatem.eis.uva.es}
\thanks{Partly supported by Ministerio de Econom\'{\i}a y Competitividad / FEDER
under project MTM2015-66330-P, by Ministerio de Ciencia, Innovaci\'{o}n y Universidades
under project RTI2018-096523-B-I00, by European Commission under project
H2020-MSCA-ITN-2014, and by INDAM -- GNAMPA Project 2018}
\date{}
\begin{abstract}
The Yakubovich Frequency Theorem, in its periodic version and in its general
nonautonomous extension, establishes conditions which are equivalent to
the global solvability of a minimization problem of infinite horizon type,
given by the integral in the positive half-line of a quadratic functional
subject to a control system. It also provides the unique minimizing pair
\lq\lq solution, control\rq\rq~and
the value of the minimum. In this paper we establish less restrictive conditions
under which the problem is partially solvable, characterize the set of
initial data for which the minimum exists, and obtain its value as well a
minimizing pair. The occurrence of exponential dichotomy and the
null character of the rotation number for a nonautonomous
linear Hamiltonian system defined
from the minimization problem are fundamental in the analysis.
\end{abstract}
\maketitle
\section{Introduction and main result}
Let us consider the control problem
\begin{equation}\label{1.control}
 \bx' = A_0(t)\,\bx+ B_0(t)\,\bu
\end{equation}
for $\bx\in\R^n$ and $\bu\in\R^m$,
the quadratic form (or supply rate)
\begin{equation}\label{1.defQ}
 \mQ(t,\bx,\bu):=\frac{1}{2}\left(\langle\,\bx, G_0(t)\,\bx\,\rangle
 +2\,\langle\,\bx,g_0(t)\,\bu\,\rangle+\langle\,\bu, R_0(t)\,\bu\,\rangle\right)\,,
\end{equation}
and a point $\bx_0\in\R^n$. We represent by $\mP_{\bx_0}$ the set of pairs
$(\bx,\bu)\colon[0,\infty)\to\R^n\!\times\R^m$ of measurable functions
satisfying \eqref{1.control} with $\bx(0)=\bx_0$, and
consider the problem of minimizing the quadratic functional
\begin{equation}\label{1.defI}
 \mI_{\bx_0}\colon\mP_{\bx_0}\to\R\cup\{\pm\infty\}\,,
 \quad(\bx,\bu)\mapsto \int_0^{\infty} \mQ(t,\bx(t),\bu(t))\,dt\,.
\end{equation}
The functions $A_0$, $B_0$, $G_0$, $g_0$, and $R_0$
are assumed to be bounded and uniformly
continuous on $\R$, with values in the sets of real matrices
of the appropriate dimensions; $G_0$ and
$R_0$ are symmetric, with $R_0(t)\ge\rho I_m$ for a common $\rho>0$ and all $t\in\R$;
and $\langle\,\cdot,\cdot\,\rangle$ represents the Euclidean
inner product in $\R^n$ or $\R^m$. A pair $(\bx,\bu)\in\mP_{\bx_0}$ is
{\em admissible\/} for $\mI_{\bx_0}$
if $(\bx,\bu)\in L^2([0,\infty),\R^n)\times L^2([0,\infty,\R^m)$. That is,
$\bu\colon[0,\infty)\to\R^m$ belongs to $L^2([0,\infty),\R^m)$,
$\bx\colon[0,\infty)\to\R^n$ solves \eqref{1.control} for this control with $\bx(0)=\bx_0$, and
$\bx$ belongs to $L^2([0,\infty),\R^n)$. In particular,
$\mI_{\bx_0}(\bx,\bu)\in\R$ if the pair $(\bx,\bu)$ is admissible.
\par
The questions we will consider are classical in control theory:
the existence of admissible pairs
and, if so, the possibility of minimizing the functional
$\mI_{\bx_0}(\bx,\bu)$ evaluated on the set of these pairs.
The existence of a finite minimum
is not a trivial question even if admissible pairs exist:
since no assumption is made on the sign of $\mQ$,
the infimum can be $-\infty$.
\par
This infinite horizon problem was considered for $T$-periodic coefficients
$A_0,B_0$, $G_0,g_0$ and $R_0$ by Yakubovich in \cite{yaku3,yaku4},
where he explains the origin of the problem and summarizes the
results previously known, providing numerous references. Under the hypothesis
of existence of at least an admissible pair for any $\bx_0\in\R^n$,
it is proved in \cite{yaku3,yaku4} the equivalence between the solvability of
the minimization problem for any $\bx_0\in\R^n$ and other seven conditions
which are
formulated in terms of the properties of a $2n$-dimensional periodic
linear Hamiltonian system which is provided by the coefficients of the minimization
problem:
\begin{equation}\label{1.hamil}
 \left[\!\begin{array}{c}\bx\\\by\end{array}\!\right]'
 =H_0(t)\,\left[\!\begin{array}{c}\bx\\\by\end{array}\!\right]
\end{equation}
with
\[
 H_0:=\left[\begin{array}{cc} A_0-B_0\,R_0^{-1}g_0^T&\;
                        B_0\,R_0^{-1}B_0^T\\
                   G_0-g_0\,R_0^{-1}g_0^T &\; -A_0^T\!+g_0\,R_0^{-1}B_0^T
                   \end{array}\right].
\]
Among all these equivalences, one of the most meaningful reads as follows:
there exists a unique admissible pair providing
the minimum value for $\mI_{\bx_0}$ for any $\bx_0\in\R^n$ if and only,
in Yakubovich's words, the {\em frequency condition\/} and the
{\em nonoscillation condition\/} are satisfied. That is,
if \eqref{1.hamil}
admits exponential dichotomy on $\R$ and, in addition,
the Lagrange plane $l^+$ composed by those initial data $\lsm\bx_0\\\by_0\rsm$ giving rise
to a solution bounded at $+\infty$ admits a basis whose vectors are the columns
of a matrix $\lsm I_n \\M^+\rsm$. Here, $I_n$ is the identity $n\times n$ matrix, and
$M^+$ turns out to be a symmetric matrix (since $l^+$ is a Lagrange plane).
In addition, the minimizing pair $(\wit\bx,\wit\bu)$ can be obtained from the
solution $\lsm \wit\bx(t)\\\wit\by(t)\rsm$ of \eqref{1.hamil} with initial data
$\lsm \bx_0\\M^+\bx_0\rsm$ via the feedback rule
\begin{equation}\label{1.frule}
 \wit\bu(t)=R_0^{-1}(t)\,B_0^T(t)\,\wit\by(t)- R_0^{-1}(t)\,g_0^T(t)\,\wit\bx(t)\,,
\end{equation}
and the value of the minimum is $\mI_{\bx_0}(\wit\bx,\wit\bu)=-(1/2)\,\bx_0^TM^+\bx_0$.
\par
In this paper we go deeper in the analysis of these problems:
we will establish conditions including
the existence of exponential dichotomy under which
it is possible to characterize the set of those $\bx_0\in\R^n$
for which there exist admissible pairs, in terms of a relation
between $\bx_0$ and $l^+$; we will check that for $\bx_0$ in this set,
the minimum is finite; and we will determine the value
of the minimum as well as an admissible pair on which it is attained.
\par
But we will not limit ourselves to the periodic case.
Yakubovich Frequency Theorem was later extended to the general nonautonomous case
of bounded and uniformly continuous coefficients: in Fabbri {\em et al.}
\cite{fajn4} six equivalences where proved in the case of recurrent coefficients;
in Johnson and N\'{u}\~{n}ez \cite{jonu2} the theorem was proved in the general
(non necessarily recurrent) case; and in Chapter 7 of Johnson {\em et al.} \cite{jonnf}
the list of eight equivalences
was completed, adding two more ones related to the rotation number.
When dealing with this general case, the problem is analyzed by including
it in a family of problems of the same type, by means of
the so called {\em hull\/} or {\em Bebutov construction\/}, which we will recall
in Section \ref{2.sechull}. This procedure provides the following families of
control systems and functionals:
\begin{align}
 \bx'=&\,A(\wt)\,\bx+B(\wt)\,\bu\,,\label{1.controlw}\\
 \mQ_\w(t,\bx,\bu):=&\,\Frac{1}{2}\left(\langle\,\bx, G(\wt)\,\bx\,\rangle
 +2\langle\,\bx,g(\wt)\,\bu\,\rangle+\langle\,\bu, R(\wt)\,\bu\,\rangle\right),\label{1.defQw}\\
 \mI_{\bx_0,\w}\colon\mP_{\bx_0,\w}&\to\R\cup\{\pm\infty\}\,,\quad
 (\bx,\bu)\mapsto\int_0^{\infty} \mQ_\w(t,\bx(t),\bu(t))\,dt\label{1.defIw}
\end{align}
for $\w\in\W$ and $\bx_0\in\R^n$.
Here, $\W$ is a compact metric space admitting a continuous flow
$\sigma\colon\R\times\W\to\W,\,(t,\w)\mapsto\sigma(t,\w)=:\wt$;
$A$, $B$, $G$, $g$, and $R$
are bounded and uniformly continuous matrix-valued functions on $\W$;
$G$ and $R$ are symmetric with $R>0$ (which ensures that $R$ is positively
bounded from below, since $\W$ is compact); and $\mP_{\bx_0,\w}$ is the set of
measurable pairs $(\bx,\bu)\colon[0,\infty)\to\R^n\!\times\R^m$ solving \eqref{1.controlw}
for $\w$ with $\bx(0)=\bx_0$. The admissible pairs are defined
for each $\w\in\W$ and $\bx_0\in\R^n$ as for the single problem \eqref{1.control}, and
their existence is guaranteed by the following hypothesis (see e.g.~Proposition
7.4 of \cite{jonnf}):
\begin{itemize}
\item[\hypertarget{H}{H}]
There exists a continuous $m\times n$ matrix-valued function
$K_0\colon\W\to\M_{m\times n}(\R)$ such that the family of linear systems
\[
 \bx'=(A(\wt)+ B(\wt)\,K_0(\wt))\,\bx\,,\qquad \w\in\W\,,
\]
is of Hurwitz type at $+\infty$.
\end{itemize}
The definition of the Hurwitz character of the family, related to the concept of
 exponential dichotomy, is given in Section \ref{sec2}.
Under this condition, the equivalences are formulated in \cite{fajn4} in
terms of the properties of the family of linear Hamiltonian systems
\begin{equation}\label{1.hamilw}
 \left[\!\begin{array}{c}\bx\\\by\end{array}\!\right]'=
 H(\wt)\,\left[\!\begin{array}{c}\bx\\\by\end{array}\!\right],\qquad \w\in\W
\end{equation}
given by
\begin{equation}\label{1.defH}
 H:=\left[\begin{array}{cc} A-B\,R^{-1}g^T&\;
                        B\,R^{-1}B^T\\
                   G-g\,R^{-1}g^T &\; -A^T\!+g\,R^{-1}B^T
                   \end{array}\right].
\end{equation}
We will use the notation \eqref{1.hamilw}$_\w$ to refer to the
system of the family corresponding to the element $\w$ of $\W$,
and we will make the same with the remaining equations defined
along the orbits of $\W$.
The results of \cite{fajn4} and \cite{jonu2} show, in particular, the
equivalence of the following situations if \hyperlink{H}{H} holds:
\begin{itemize}
\item[(1)]
The family of linear Hamiltonian systems
\eqref{1.hamilw} admits an exponential dichotomy
over $\W$ and the (symmetric) Weyl matrix-valued function $M^+$ globally exists;
that is, each one of the systems of the family
admits an exponential dichotomy on $\R$ and
for any $\w\in\W$ the Lagrange plane $l^+(\w)$ of the solutions
bounded at $+\infty$ admits as basis the column vectors of a
matrix $\lsm I_n\\M^+(\w)\rsm$ (see Section \ref{sec2}).
\item[(2)] The minimization
problem is solvable for each $\w\in\W$ and $\bx_0\in\R^n$.
\end{itemize}
In addition, in this case
the minimizing pair $(\wit\bx_\w,\wit\bu_\w)$ comes
from the solution $\lsm\wit\bx_\w(t)\\\wit\by_\w(t)\rsm$
of \eqref{1.hamilw}$_\w$ with initial data
$\lsm \bx_0\\M^+(\w)\,\bx_0\rsm$ via the feedback rule
\begin{equation}\label{1.frulew}
 \wit\bu_\w(t)= R^{-1}(\wt)\,B^T(\wt)\,\wit\by_\w(t)- R^{-1}(\wt)\,g^T(\wt)\,\wit\bx_\w(t)\,,
\end{equation}
and the value of the minimum is
$\mI_{\bx_0,\w}(\wit\bx_\w,\wit\bu_\w)=-(1/2)\,\bx_0^TM^+(\w)\,\bx_0$.
And, in fact, \hyperlink{H}{H} holds when there is exponential dichotomy and
$M^+$ globally exists. (We point out that
the rule \eqref{1.frulew}$_\w$ provides a pair
\lq\lq state, control\rq\rq~solving \eqref{1.controlw}$_\w$
whenever we have a solution of \eqref{1.hamilw}$_\w$.)
\par
Among the remaining equivalences, we want to call attention to another
one, formulated in terms of the rotation number of the family \eqref{1.hamilw},
and which holds when $\W$ admits an ergodic measure $m$ with total support. If so,
and always under hypothesis \hyperlink{H}{H}, the previous situation is equivalent
to
\begin{itemize}
\item[(3)] the rotation number of the family \eqref{1.hamilw} for $m$ is zero.
\end{itemize}
\par
Now we will formulate our main result. It is fundamental to note that
hypothesis \hyperlink{H}{H} is not
in force: otherwise the assumptions of the theorem would imply the
global solvability of the family of minimization problems.
Recall that we have represented
by $l^+(\w)$ the Lagrange plane of the solutions bounded at $+\infty$
in the case of exponential dichotomy over $\W$ of the family \eqref{1.hamilw}.
And recall also that the conditions assumed on $A,B,G,g$ and
$R$ (described after \eqref{1.defIw}) are in force.
\begin{teor}\label{1.teormain}
Let us assume that $\W$ is minimal, that there exists $\w_0\in\W$ such that the
$n\times m$ matrix $B(\w_0)$ has full rank, that the family of systems
\eqref{1.hamilw} admits exponential dichotomy over $\W$, and that
there exists a $\sigma$-ergodic measure on $\W$
for which the corresponding rotation number is zero. Let $l^+(\w)$
be the Lagrange plane of the solutions of \eqref{1.controlw}$_\w$
bounded at $+\infty$. And let us fix $\w\in\W$ and $\bx_0\in\R^n$. Then,
\begin{itemize}
\item[\rm(i)] there exist admissible pairs $(\bx,\bu)$ for the functional
$\mI_{\bx_0,\w}$ given by \eqref{1.defIw}$_\w$
if and only if there exists $\by_0\in\R^n$ such that
$\lsm\bx_0\\\by_0\rsm\in l^+(\w)$.
\item[\rm(ii)] In this case, the infimum of $\mI_{\bx_0,\w}$ is finite.
In addition, if the columns of the
$2n\times n$ matrix $\lsm L_{\w,1}\\L_{\w,2}\rsm$ are a basis of $l^+(\w)$ and
$\lsm\bx_0\\\by_0\rsm=\lsm L_{\w,1}\,\bc\\L_{\w,2}\,\bc\rsm$ for a vector $\bc\in\R^n$, then
the infimum is given by $-(1/2)\,\bc^T L^T_{\w,2}\,L_{\w,1}\,\bc$,
and a minimizing pair $(\wit\bx_\w,\wit\bu_\w)\in\mP_{\bx_0,\w}$ is obtained from the solution
$\lsm\wit\bx_\w(t)\\\wit\by_\w(t)\rsm$ of \eqref{1.hamilw}$_\w$ with initial data
$\lsm\bx_0\\\by_0\rsm$ via the feedback rule \eqref{1.frulew}$_\w$.
\item[\rm(iii)] If the situation in {\rm (i)} does not hold, then $\mI_{\bx_0}(\bx,\bu)=\infty$ for
any pair $(\bx,\bu)\in\mP_{\bx_0,\w}$.
\end{itemize}
\end{teor}
\par
Section \ref{sec2} contains the notions and basic properties required to
fully understand the hypotheses and statements of Theorem~\ref{1.teormain},
whose proof is given in Section~\ref{sec3}. In that section we will also analyze the
hypotheses; we will explain how these hypotheses can be formulated for the initial
problem, for which we give a less general version of the main theorem; and we will
show autonomous and nonautonomous scenarios in which admissible pairs do not
always exist.
\section{Preliminaries}\label{sec2}
All the contents of this preliminary section can be found in
Johnson {\em et al.}~\cite{jonnf}, together with a quite exhaustive list of references
for the origin of the results here summarized.
\par
Let us fix some notation. The set
of $d\times m$ matrices with entries in the real line $\R$
is represented by $\M_{d\times m}(\R)$.
As usual, $\R^d:=\M_{d\times 1}(\R)$, and
$A^T$ is the transpose of the matrix $A$. The subset
$\s_d(\R)\subset\M_{d\times d}(\R)$ is composed by the symmetric matrices.
The expressions $M>0$, $M\ge 0$, $M<0$, and $M\le 0$ for $M\in\s_d(\R)$
mean that $M$ is positive definite, positive semidefinite,
negative definite, and negative semidefinite, respectively.
If $M\colon\W\to\s_d(\R)$ is a map, $M>0$ means that $M(\w)>0$ for all
the elements $\w\in\W$, and $M<0$, $M\ge 0$, and $M\le 0$ have the
analogous meaning.
It is also obvious what $M_1>M_2$, $M_1\ge M_2$,
$M_1<M_2$, and $M_1\le M_2$ mean. We represent by $I_d$ and $0_d$ the
identity and zero $d\times d$ matrices, by $\bcero$ the null
vector of $\R^d$ for all $d$, and
by $\n{\cdot}$ the Euclidean norm in $\R^d$.
\par
A real {\em Lagrange plane\/} is an $n$-dimensional
linear subspace of $\R^{2n}$ such that
$[\bx_1^T\,\by_1^T]\,J\,\lsm\bx_2\\\by_2\rsm=0$ for any pair
of its elements $\lsm\bx_1\\\by_1\rsm$ and $\lsm\bx_2\\\by_2\rsm$,
where $J:=\lsm 0_n&-I_n\\I_n&0_n\rsm$.
A Lagrange plane $l$ {\em is represented by\/} $\lsm L_1\\ L_2\rsm$
(which we represent as $l\equiv \lsm L_1\\L_2\rsm$)
if the column vectors of the matrix form a basis of the $n$-dimensional linear
space $l$. In this case $L_2^TL_1=L_1^TL_2$. Note that it can be also represented
by $\lsm I_n\\M\rsm$ if and only if $\det L_1\ne 0$, in which case the matrix
$M=L_2L_1^{-1}$ is symmetric.
\par
The next concepts and properties are basic in
topological dynamics and measure theory.
A ({\em real and continuous})
{\em global flow\/} on a complete metric space $\W$
is a continuous map
$\sigma\colon\R\times\W\to\W,\; (t,\w)\mapsto\sigma(t,\w)$
such that
$\sigma_0=\text{Id}$ and $\sigma_{s+t}=\sigma_t\circ\sigma_s$ for each $s,t\in\R$,
where $\sigma_t(\w)=\sigma(t,\w)$.
The $\sigma$-{\em orbit\/} of a point $\w\in\W$
is the set $\{\sigma_t(\w)\,|\;t\in\R\}$.
A subset $\W_1\subset \W$ is {\em $\sigma$-invariant\/}
if $\sigma_t(\W_1)=\W_1$ for every $t\in\R$.
A $\sigma$-invariant subset $\W_1\subset\W$ is {\em minimal\/} if it is compact
and does not contain properly any other compact $\sigma$-invariant set.
And the continuous flow $(\W,\sigma)$ is
{\em minimal\/} if $\W$ itself is minimal.
\par
Let $m$ be a normalized Borel measure on $\W$; i.e.~a finite
regular measure defined on the Borel subsets of $\W$ and with
$m(\W)=1$. The measure $m$ is {\em $\sigma$-invariant\/}
if $m(\sigma_t(\W_1))=m(\W_1)$ for every Borel subset
$\W_1\subset\W$ and every $t\in\R$. A $\sigma$-invariant measure $m$
is {\em $\sigma$-ergodic\/} if
$m(\W_1)=0$ or $m(\W_1)=1$ for every $\sigma$-invariant subset
$\W_1\subset\W$. A real continuous flow $(\W,\sigma)$
admits at least an ergodic measure whenever $\W$ is compact.
And the {\em topological support\/} of $m$, $\Supp m$, is the complement
of the largest open set $O\subset\W$ for which $m(O)=0$.
We say that $m$ {\em has total support\/} if its topological support is $\W$.
If $\W$ is minimal, then any $\sigma$-ergodic
measure has total support.
\par
In the rest of the paper, $(\W,\sigma)$ will be a minimal continuous
global flow on a compact metric space, and we will
denote $\wt:=\sigma(t,\w)$. We will work with families of linear
systems
of the type \eqref{1.hamilw} depending on continuous matrix-valued functions
$A,B,G,g$ and $R$ with the properties defined in the Introduction.
Since $H\colon\W\to\mathfrak{sp}(n,\R)$, where
\[
 \mathfrak{sp}(n,\R):=\{H\in\M_{2n\times 2n}(\R)\,|\;H^T\!J+JH=0_{2n}\}
\]
for $J=\lsm 0_n&-I_n\\I_n&0_n\rsm$, the systems of the family are
said to be Hamiltonian.
Let $U_H(t,\w)$ denote the fundamental matrix solution of the system
\eqref{1.hamilw}$_\w$ with $U(0,\w)=I_{2n}$.
The family \eqref{1.hamilw} induces
a real continuous global flow
on the linear bundle $\W\times\R^{2n}$,
\begin{equation}\label{2.deftau}
 \tau_H\colon\R\times\W\times\R^{2n}\to\W\times\R^{2n}\,,\quad
 \big(t,\w,\lsm\bx\\\by\rsm\big)\mapsto\big(\wt,U_H(t,\w)\,\lsm\bx\\\by\rsm\big)\,.
\end{equation}
An equivalent assertion can be done for any family of linear systems
\begin{equation}\label{2.linw}
 \bw'=S(\wt)\,\bw\,,\qquad \w\in\W
\end{equation}
for $\bw\in\R^d$ if $S\colon\W\to\M_{d\times d}(\R)$ is continuous.
We represent the corresponding fundamental matrix solution as
$U_S\colon\R\times\W\to\M_{d\times d}(\R)$, and  the flow that it provides as
$\tau_S\colon\R\times\W\times\R^d\to\W\times\R^d$.
\par
In the rest of this section we recall some basic concepts and some
associated properties related to families of the forms~\eqref{2.linw} and
\eqref{1.hamilw}. Some of them are directly related to the
statements of our main result (as in the case of the
{\em exponential dichotomy\/} and the {\em rotation number}),
while other ones are used as tools in its proof, as it happens with
the {\em uniform weak disconjugacy} property).
\par
We begin with the exponential dichotomy of a family of linear systems over $\W$,
which is one of the main hypotheses of our main theorem.
\begin{defi}\label{2.defED}
The family~\eqref{2.linw}
has {\em exponential dichotomy\/}
{\em over\/} $\W$ if there exist
constants $\eta\ge 1$ and $\beta>0$ and a splitting $\W\times
\R^d=L^+\oplus L^-$ of the bundle into the Whitney sum of
two closed subbundles such that
\begin{itemize}
 \item[-] $L^+$ and $L^-$ are invariant
 under the flow $\tau_S$ induced by \eqref{2.linw} on $\W\times\R^d$;
 that is, if $(\w,\bw)$ belongs to $L^+$ (or to $L^-$), so does
 $(\wt,U_S(t,\w)\,\bw)$ for all $t\in\R$ and $\w\in\W$.
 \item[-] $\n{U_S(t,\w)\,\bw} \le \eta\,e^{-\beta t}\n{\bw}\quad$ for
   every $t\ge 0$ and $(\w,\bw)\in L^+$.
 \item[-] $\n{U_S(t,\w)\,\bw} \le \eta\,e^{\beta t}\n{\bw}\quad\;\;\,$
   for every $t\le 0$ and $(\w,\bw)\in L^-$.
\end{itemize}
In the case that $L^+=\R^d$, the family \eqref{2.linw} is
{\em of Hurwitz type at} $+\infty$.
\end{defi}
In general, we will omit the words \lq\lq over $\W$" when the family~\eqref{2.linw}
has exponential dichotomy, since no confusion arises.
Let us summarize in the next list of remarks some
well-known fundamental properties
satisfied by a family of linear systems which has exponential dichotomy.
\begin{notas}\label{2.notasED}
1.~If $\W$ is minimal (as we assume in Theorem~\ref{1.teormain}),
the exponential dichotomy of
the family \eqref{2.linw} over $\W$ is equivalent to the exponential dichotomy
on $\R$ of anyone of its systems (see e.g.~\cite{copp1} for the definition of this
classical concept). This property is proved in \cite{sase2}
(Theorem 2 and Section 3). In addition, the exponential dichotomy
of the family is equivalent to the unboundedness of any nontrivial
solution of anyone of the systems, as proved in
Theorem 10.2 of \cite{selg}.
\par
2.~Suppose that the family \eqref{2.linw} has
exponential dichotomy. There exists $\delta>0$ such that
if $T\colon\W\to\M_{d\times d}(\R)$ is a continuous map and
$\max_{\w\in\W}\n{S(\w)-T(\w)}<\delta$ (where $\n{\cdot}$ is
representing the Euclidean operator norm), then the family
$\bw'=T(\wt)\,\bw,\;\w\in\W$ has exponential dichotomy.
This well-known property of robustness is a consequence of the Sacker
and Sell Spectral Theorem (Theorem 6 of \cite{sase5}).
\par
3.~Assume that the family of linear systems is of Hamiltonian type, as in the case
of \eqref{1.hamilw}, and that it has exponential dichotomy. Then the sections
\begin{equation}\label{2.defl}
 l^\pm(\w):=\big\{\lsm\bx\\\by\rsm\in\R^{2n}\,\big|\;\big(\w,\lsm\bx\\\by\rsm\big)\in L^\pm\big\}
\end{equation}
are real Lagrange planes.
In addition,
\[
\begin{split}
 l^\pm(\w)&=\{\bz\in\R^{2n}\,|\;\lim_{t\to\pm\infty}\n{U_H(t,\w)\,\lsm\bx\\\by\rsm}=\bcero\}\\
 &=\{\bz\in\R^{2n}\,|\;\sup_{\pm t\in[0,\infty)}\n{U_H(t,\w)\,\lsm\bx\\\by\rsm}<\infty\}\,.
\end{split}
\]
These properties are proved, for example, in Section 1.4 of \cite{jonnf}.
\par
4.~Also in the Hamiltonian case, assume that for all $\w\in\W$ the Lagrange plane
$l^+(\w)$ can be represented by the matrix $\lsm I_n\\M^+(\w)\rsm$.
Or, equivalently, that for all $\w\in\W$ the Lagrange plane
$l^+(\w)$ can be represented by a matrix $\lsm L_{\w,1}^+\\L_{\w,2}^+\rsm$
with $\det L_{\w,1}^+\ne 0$ (so that $M^+(\w)=L_{\w,2}\,L_{\w,1}^{-1}$).
In this case $M^+\colon\W\to\s_n(\R)$ is a continuous matrix-valued
function, and it is known as one of the {\em Weyl functions}
for \eqref{1.hamilw}. In this situation, we say that
the Weyl function $M^+$~{\em globally exists}.
(The other Weyl function is $M^-$, associated
to the subbundle $L^-$, and it satisfies the same properties if it exists.)
\end{notas}
The other fundamental hypotheses of our main theorem refers to the value
of the rotation number for the family \eqref{1.hamilw}, whose definition we recall now.
This object depends on a given $\sigma$-ergodic measure on $\W$. Among
the many equivalent definitions
for this quantity, we give one which extends that which is possibly the best known
in dimension 2 (see \cite{jomo} and \cite{gijo}). We write as $U_H(t,\w)=\lsm
U_1(t,\w)&U_3(t,\w)\\ U_2(t,\w)&U_4(t,\w)\rsm$ the matrix-valued
solution of \eqref{1.hamilw} with $U_H(0,\w)=I_{2n}$. And $\arg\colon\C\to\R$
is the continuous branch of the argument of a complex number for which
$\arg 1=0$.
\begin{defi}\label{2.defrot}
Let $m$ be a $\sigma$-ergodic measure on $\W$.
The {\em rotation number $\alpha(m)$ of the family\/}
of linear Hamiltonian systems \eqref{1.hamilw}
{\em with respect to $m$} is the value of
\[
 \lim_{t\to\infty}\frac{1}{t}\:\arg\det (U_1(t,\w)-iU_2(t,\w))
\]
for $m$-a.a.~$\w\in\W$.
\end{defi}
It is proved in \cite{nono} that the limits exist and take the same finite
value for $m$-a.a.~$\w\in\W$. The analysis of $\alpha(m)$
made in \cite{nono} is completed in \cite{fajn1} and
in Chapter 2 of \cite{jonnf}, where the interested reader may find
many other equivalent definitions and a exhaustive description of the properties
of the rotation number.
\par
Now we introduce the concept of uniform weak disconjugacy.
\begin{defi}
{\rm \label{2.defUWD}
The family of linear Hamiltonian systems \eqref{1.hamilw} is
{\em uniformly weakly disconjugate on\/} $[0,\infty)$
if there exists $t_0\ge 0$ independent of $\w$ such that for every
nonzero solution $\bz(t,\w)=\lsm \bz_1(t,\w)\\\bz_2(t,\w)\rsm$ of the
systems corresponding to $\w$
with $\bz_1(0,\w)=\bcero$, there holds $\bz_1(t,\w)\ne\bcero$ for all
$t>t_0$.}
\end{defi}
In the next remarks some results proved in \cite{jnno}, \cite{fajn5} and
in Chapter 5 of \cite{jonnf} are summarized.
Note that the fact that the submatrix $B\,R^{-1}B^T$ of $H$
(see \eqref{1.defH}) is positive
semidefinite is fundamental in what follows.
\begin{notas}\label{2.notasUWD}
1.~If the family \eqref{1.hamilw} is uniformly weakly disconjugate, then there exist
{\em uniform principal solutions at $\pm\infty$},
$\lsm L_1^\pm(t,\w)\\L_2^\pm(t,\w)\rsm$ (see Theorem 5.2 of \cite{jnno} or
Theorem 5.26 of \cite{jonnf}).
They are real $2n\times n$ matrix-valued solutions of \eqref{1.hamilw} satisfying the
following properties: for all $t\in\R$ and $\w\in\W$,
the matrices $L_1^\pm(t,\w)$ are nonsingular
and $\lsm L_1^\pm(t,\w)\\L_2^\pm(t,\w)\rsm$ represent Lagrange planes;
and for all $\w\in\W$,
\[
 \lim_{\pm t\to\infty}\left(\int_0^t
 (L_1^\pm)^{-1}(s,\w)\,H_3(\ws)\, ((L_1^\pm)^T)^{-1}(s,\w)\,ds\right)^{-1}\!=0_n\,.
\]
In addition, if the matrix-valued functions
$\lsm L_1^\pm(t,\w)\\L_2^\pm(t,\w)\rsm$ are uniform principal solutions
at $\pm\infty$, then the real matrix-valued functions
$N^\pm\colon\W\to\s_n(\R)\,,\;\w\mapsto N^\pm(\w):=
L_2^\pm(0,\w)\,(L_1^\pm(0,\w))^{-1}$ are unique. The
functions $N^\pm$  are called {\em principal functions of\/} \eqref{1.hamilw}.
\par
The interested reader can find in Chapter 5 of \cite{jonnf} a careful
description of the uniform principal solutions and the principal
functions for families of linear Hamiltonian systems of the type \eqref{1.hamilw}.
A more general theory concerning principal solutions for less restrictive
assumptions is developed in \cite{sesi12, sesi14} and references therein.
And a recent in-deep analysis of the corresponding Riccati equations (which the
principal functions solve) can be found in \cite{sept4}.
\par
2.~According to Theorem 5.2 of \cite{jnno} or to Theorem 5.17 of \cite{jonnf},
the uniform weak disconjugacy of the family \eqref{1.hamilw} ensures
the validity of the condition
\begin{itemize}
\item[\bf D2] \hypertarget{D2} For all $\w\in\W$ and for any nonzero solution
$\lsm \bx(t,\w)\\\by(t,\w)\rsm$ of the system \eqref{1.hamilw}$_\w$
with $\bx(0,\w)=\bcero$, the vector $\bx(t,\w)$
does not vanish identically on $[0,\infty)$.
\end{itemize}
It also ensures that the rotation number $\alpha(m)$ with respect to any ergodic
measure $m$ on $\W$ vanishes: see Theorem 2 of \cite{fajn5} or Theorem 5.67 of \cite{jonnf}.
Conversely, if there exists a $\sigma$-ergodic measure on $\W$ with total support
(which is always the case if $\W$ is minimal)
for which the rotation number is zero and \hyperlink{D2}{D2} holds, then
the family \eqref{1.hamilw} is uniformly weakly disconjugate. This assertion
follows from Theorems 5.67 and 5.17 of \cite{jonnf}.
\end{notas}
\subsection{The hull construction}\label{2.sechull}
Let us complete Section \ref{sec2} by explaining briefly how we obtain the family of
problems given by \eqref{1.controlw}, \eqref{1.defQw} and \eqref{1.defIw}
from the initial
one, given by \eqref{1.control}, \eqref{1.defQ} and \eqref{1.defI}.
\par
Let us denote $C_0:=(A_0,B_0,G_0,g_0,R_0)$, so that
\begin{equation}\label{2.defC0}
 C_0\colon\R\to\M_{n\times n}(\R)\times \M_{n\times m}(\R)
 \times \M_{n\times n}(\R)\times \M_{n\times m}(\R)\times \M_{m\times m}(\R)\,,
\end{equation}
and define $\W$ as its {\em hull\/}: that is, the closure with respect to
the compact-open topology of $\R$ of the set $\{C_s\,|\;s\in\R\}$, where
$C_s(t):=C_0(t+s)$.
It turns out that $\W$ is a compact metric space and that the time translation map
\[
 \sigma\colon\R\times\W\to\W\,,\quad(t,\w)\mapsto\wt\,,
\]
where $(\wt)(s)=\w(t+s)$ defines a continuous flow on $\W$.
The proofs of these assertions can be found in Sell \cite{sell2}.
\par
Note that any element $\w\in\W$ can be written as
$\w=(\w_1,\w_2,\w_3,\w_4,\w_5)$, and that
$\wt=(\w_1{\cdot}t,\w_2{\cdot}t,\w_3{\cdot}t,\w_4{\cdot}t,\w_5{\cdot}t)$
with $(\w_i{\cdot}t)(s)=\w_i(t+s)$ for $i=1,\ldots,5$. We define
\[
 A\colon\W\to\M_{n\times n}(\R)\,,\quad(\w_1,\w_2,\w_3,\w_4,\w_5)\mapsto\w_1(0)\,,
\]
and proceed in a similar way to define $B\colon\W\to\M_{n\times m}(\R)$,
$G\colon\W\to\M_{n\times n}(\R)$, $g\colon\W\to\M_{n\times m}(\R)$, and
$R\colon\W\to\M_{n\times n}(\R)$. It is obvious that
$A,B,G,g$ and $R$ are continuous maps on $\W$.
In addition, if $\wit\w=C_0\in\W$, then $A(\wit\w{\cdot}t)=
(\wit\w_1{\cdot}t)(0)=\wit\w_1(t)=A_0(t)$, and analogous
equalities hold for $B,G,g$ and $R$. This means that the family of problems
given by \eqref{1.control}, \eqref{1.defQ} and \eqref{1.defI} for $\w\in\W$
includes the initial one, which corresponds to the element $C_0$ of $\W$.
Note that $G$ and $R$ are symmetric, and that $R>0$.
\par
Additional recurrence properties on $C_0$ ensure that the flow on
$\W$ is minimal, which is one of the hypotheses of Theorem \ref{1.teormain}.
This is for instance the case when $C_0$ is almost periodic or
almost automorphic (see e.g.~\cite{fink} and \cite{shyi4} for the definitions).
Note that in the minimal case, $\W$ is the hull of
any of its elements.
\section{Proof of Theorem \ref{1.teormain}}\label{sec3}
Let $(\W,\sigma)$ be a real continuous global flow on a
compact metric space, and let us denote $\wt:=\sigma(t,\w)$.
We consider the family of control systems \eqref{1.controlw}
and functionals \eqref{1.defIw} under the conditions on the coefficients
$A,B,G,g$ and $R$ described in the Introduction (which are guaranteed
under the initial conditions on $A_0,B_0,G_0,g_0$ and $R_0$, as explained in
Section \ref{2.sechull}), and consider the minimization
problem there explained. We also consider the family
of linear Hamiltonian systems defined by \eqref{1.hamilw} and \eqref{1.defH}.
During this whole section, we will be working under the hypotheses of Theorem
\ref{1.teormain}, namely
\begin{hipos}\label{3.hipos}
$\W$ is minimal, there exists $\w_0\in\W$ such that the $n\times m$ matrix $B(\w_0)$
has full rank,
the family of linear Hamiltonian systems \eqref{1.hamilw}
has exponential dichotomy
over $\W$, and there exists a $\sigma$-ergodic measure $m$ on $\W$
for which the rotation number $\alpha(m)$ is zero.
\end{hipos}
We will analyze later on the scope of these hypotheses. Let us begin
with a result which includes the assertions of Theorem \ref{1.teormain}
in the simplest situation, that of $m\ge n$. This result will play a
fundamental role in the general proof.
\begin{teor}\label{3.teorBregular}
Assume that Hypotheses~\ref{3.hipos} hold and that
$m\ge n$. Then,
\begin{itemize}
\item[\rm(i)] the family \eqref{1.hamilw} is uniformly
weakly disconjugate.
\item[\rm(ii)] The Weyl functions $M^\pm\colon\W\to\s_n(\R)$
associated to the exponential dichotomy of the family \eqref{1.hamilw}
globally exist, and they agree with the principal functions $N^\pm$.
\item[\rm(iii)] There exist admissible pairs for the functional
$\mI_{\bx_0,\w}$ given by \eqref{1.defIw}$_\w$
for all $(\bx_0,\w)\in\R^n\!\times\W$,
and the corresponding minimization problem
is solvable. In addition, the (unique)
minimizing pair $(\wit\bx_\w(t),\wit\bu_\w(t))\in\mP_{\bx_0,\w}$
comes from the solution $\lsm\wit\bx_\w(t)\\\wit\by_\w(t)\rsm$
of \eqref{1.hamilw}$_\w$ with initial data
$\lsm \bx_0\\M^+(\w)\,\bx_0\rsm$ via the feedback rule \eqref{1.frulew}$_\w$,
and the value of the minimum is
$\mI_{\bx_0,\w}(\wit\bx_\w,\wit\bu_\w)=-(1/2)\,\bx_0^TM^+(\w)\,\bx_0$.
\end{itemize}
\end{teor}
\begin{proof}
Let $U_A(t,\w_0)$ be the fundamental matrix solution of $\bx'=A(\w_0{\cdot}t)\,\bx$
with $U_A(0,\w_0)=I_n$. Since the rank of the $n\times m$ matrix $B(\w_0)$ is $n$
(as one of the Hypotheses \ref{3.hipos} guarantees), the
$n\times n$ matrix $B(\w_0)\,B^T(\w_0)$ is not singular. Hence
\[
 \int_0^\infty U_A^{-1}(t,\w_0)\,B(\w_0{\cdot}t)\,B^T(\w_0{\cdot}t)\,(U_A^{-1})^T(t,\w_0)\,dt>0\,,
\]
which ensures that the system \eqref{1.controlw}$_{\w_0}$ is null
controllable (see~\cite{cont}, Theorem 7.2.2). That is, for any $\bx_0\in\R^n$
there exists a time $t_0=t_0(\bx_0,\w)$ and an integrable control $\bu\colon[0,t_0]\to\R^n$
such that the solution of the corresponding
system with $\bx(0)=\bx_0$ satisfies $\bx(t_0)=\bcero$,
which means that any $\bx_0$ can be steered to $\bcero$ in finite time
by an integrable control $\bu$.
\par
According to the results of \cite{jone0} (see also Theorem 6.4 of
\cite{jonnf}),
the minimality of $\W$ and the previous property ensure the
uniform null controllability of the family \eqref{1.controlw}; that is,
all the systems are null controllable and there is a common time $t_0(\bx_0,\w)$
for all $(\bx_0,\w)\in\R^n\times\W$.
As explained in Remark 6.8.1 of \cite{jonnf},
this uniform null controllability
holds if and only if the family \eqref{1.hamilw} satisfies condition
\hyperlink{D2}{D2}. On the other hand,
Hypotheses \ref{3.hipos} ensure the existence of
an ergodic measure $m$ for which the rotation number vanishes;
and, as said in Section \ref{sec2},
the minimality of the set $\W$ ensures that $m$ has total support.
In this situation, according to Remark~\ref{2.notasUWD}.2,
the family \eqref{1.hamilw} is uniformly weakly disconjugate, which proves (i).
The simultaneous occurrence of uniform weak disconjugacy and
exponential dichotomy ensures (ii): Theorem 5.58 of \cite{jonnf} proves the global
existence of the Weyl functions $M^\pm$,
which agree with the principal functions (see Remark \ref{2.notasUWD}.1).
Finally, as recalled in the Introduction, and according to Theorem 4.3 of
\cite{fajn4} (see also Remark 7.7 and Theorem 7.10 of \cite{jonnf}),
the presence of exponential dichotomy and the
global existence of $M^+$ ensure the assertions in~(iii).
\end{proof}
\begin{nota}\label{3.notainc}
Note that in the situation of the previous theorem, the global existence of $M^+$
ensures that if $l^+(\w)\equiv\lsm L_{\w,1}\\L_{\w,2}\rsm$, then $M^+(\w)=L_{\w,2}\,(L_{\w,1})^{-1}$
and that, for every $\bx_0\in\R^n$ there exists a unique $\bc\in\R^n$
such that $\bx_0=L_{\w,1}\,\bc$ and hence a unique $\by_0$ such that
$\lsm\bx_0\\\by_0\rsm\in l^+(\w)$, being $\by_0=L_{\w,2}\,\bc=M^+(\w)\,\bx_0$,
Note also that $\bx_0^TM^+(\w)\,\bx_0=\bc^T L_{\w,1}^TL_{\w,2}\,(L_{\w,1})^{-1}L_{\w,1}\,\bc=
\bc^TL_{\w,1}^TL_{\w,2}\,\bc=\bc^TL_{\w,2}^TL_{\w,1}\,\bc$.
These are the reasons for which we asserted that Theorem \ref{3.teorBregular}
proves Theorem~\ref{1.teormain} if $m\ge n$: under its hypotheses,
every functional $\mI_{\bx_0,\w}$ can be minimized.
\end{nota}
The next technical lemma will also be used in the proof of Theorem~\ref{1.teormain}.
Note that Hypotheses~\ref{3.hipos} are not required.
\begin{lema}\label{3.lemaV}
Let us fix $\w\in\W$, and
let $\lsm L_1(t)\\L_2(t)\rsm$ be a $2n\times n$ solution of the linear
Hamiltonian system \eqref{1.hamilw}$_\w$. Let us also fix $\bc\in\R^n$ and define
\[
\begin{split}
 \bx(t)&:=L_1(t)\,\bc\,,\quad\by(t):=L_2(t)\,\bc\,,\\
 \bu(t)&:=R^{-1}(\wt)\,B^T(\wt)\,\by(t)-R^{-1}(\wt)\,g^T(\wt)\,\bx(t)\,,\\
 V(t)&:=\by^T(t)\,\bx(t)\,,
\end{split}
\]
and $\mQ$ by \eqref{1.defQw}. Then,
\[
 \frac{d}{dt}\,V(t)=2\,\mQ_\w(t,\bx(t),\bu(t))\,.
\]
\end{lema}
\begin{proof}
It follows from the definitions in the statement that
\[
\begin{split}
 \bu(t)&=\big(R^{-1}(\wt)\,B^T(\wt)\,L_2(t)-R^{-1}(\wt)\,g^T(\wt)\,L_1(t)\big)\,\bc\,,\\
 V(t)&=\bc^TL_2^T(t)\,L_1(t)\,\bc\,.
\end{split}
\]
A straightforward and simple computation shows that
\[
\begin{split}
 \frac{d}{dt}\,V(t)=\bc^T\Big(&L_1^T(t)\,G(\wt)\,L_1(t)-L_1^T(t)\,g(\wt)\,R^{-1}(\wt)\,g^T(\wt)\,L_1(t)\\
 &+L_2(t)\,B(\wt)\,R^{-1}(\wt)\,B^T(\wt)\,L_2^T(t)\Big)\,\bc\,.
\end{split}
\]
And a longer computation shows that $\mQ_\w$ evaluated at $(t,\bx(t),\bu(t))$, namely
\[
 2\,\mQ_\w\Big(t,\,L_1(t)\,\bc,
 \,\big(R^{-1}(\wt)\,B^T(\wt)\,L_2(t)-R^{-1}(\wt)\,g^T(\wt)\,L_1(t)\big)\,\bc\Big),
\]
takes the same value.
\end{proof}
\noindent{\em Proof of Theorem~\ref{1.teormain}}.~Note that we can assume that $m<n$, since otherwise
Theorem~\ref{3.teorBregular} proves all the assertions (see also Remark~\ref{3.notainc}).
We first prove the result in a particular case, from where the general
one will follow easily.
\smallskip\par
{\sc Particular case}. Let us first assume that
$B(\w_0)=\lsm B_1(\w_0)\\B_2(\w_0)\rsm$, where $B_1(\w_0)$ is an $m\times m$
nonsingular matrix. We define the $n\times n$ matrix valued functions
\[
 B_\ep:=\left[\begin{array}{cc}B_1&0\\B_2&\ep\,I_{n-m}\end{array}\right],
 \quad  R_\ep:=\left[\begin{array}{cc}R&0\\0&\ep\,I_{n-m}\end{array}\right],
 \quad\text{and}\quad
  g_\ep:=\left[\begin{array}{cc}g&0\end{array}\right]
\]
for $\ep>0$,
where $0$ stands for the null matrix of the suitable dimension whenever it appears.
Let us consider the families of control problems and functionals given by
\begin{align}
 \bx'=&\,A(\wt)\,\bx+B_\ep(\wt)\,\bv\,,\label{3.controlwep}\\
 \mQ^\ep_\w(t,\bx,\bv):=&\,\Frac{1}{2}\left(\langle\,\bx, G(\wt)\,\bx\,\rangle
 +2\langle\,\bx,g_\ep(\wt)\,\bu\,\rangle+\langle\,\bv, R_\ep(\wt)\,\bv\,\rangle\right),\label{3.defQwep}\\
 \mI^\ep_{\bx_0,\w}\colon\mP^\ep_{\bx_0,\w}&\to\R\cup\{\pm\infty\}\,,\quad
 (\bx,\bv)\mapsto
 \int_0^{\infty} \mQ^\ep_\w(t,\bx(t),\bv(t))\,dt\,,\label{3.defIwep}
\end{align}
where $\mP^\ep_{\bx_0,\w}$ is the set of
measurable pairs $(\bx,\bv)\colon[0,\infty)\to\R^n\!\times\R^n$ solving \eqref{3.controlwep}$_\w$
with $\bx(0)=\bx_0$.
Note that both the state $\bx$ and the control $\bv$ are now $n$-dimensional.
The associated family of linear Hamiltonian systems,
\begin{equation}\label{3.hamilep}
 \left[\!\begin{array}{c}\bx\\\by\end{array}\!\right]'
 =H_\ep(\wt)\,\left[\!\begin{array}{c}\bx\\\by\end{array}\!\right]\,,\qquad \w\in\W\,,
\end{equation}
is given by
\[
\begin{split}
 H_\ep:=&\,\left[\begin{array}{cc} A-B_\ep\,R_\ep^{-1}g_\ep^T&\;
                        B_\ep\,R_\ep^{-1}B_\ep^T\\
                   G-g_\ep\,R_\ep^{-1}g_\ep^T &\; -A_0^T\!+g_\ep\,R_\ep^{-1}B_\ep^T
                   \end{array}\right]\\[.2cm]
  =&\left[\begin{array}{cc} A-B\,R^{-1}g^T&\;
                        B_\ep\,R_\ep^{-1}B_\ep^T\\
                   G-g\,R^{-1}g^T &\; -A_0^T\!+g\,R^{-1}B^T
                   \end{array}\right].
\end{split}
\]
The unique submatrix depending on $\ep$ is
\begin{equation}\label{3.Bep}
 B_\ep\,R_\ep^{-1}B_\ep^T=B\,R^{-1}B^T+\ep\left[\begin{array}{cc}0_m&0\\0&I_{n-m}
 \end{array}\right].
\end{equation}
Therefore \eqref{3.hamilep} agrees with \eqref{1.hamilw} for $\ep=0$, and hence
according to Hypotheses \eqref{3.hipos} it has exponential dichotomy over $\W$.
The robustness of this property (see Remark~\ref{2.notasED}.2)
ensures the exponential dichotomy
for the family \eqref{3.hamilep} if $\ep\in[0,\ep_0]$ for $\ep_0>0$ small enough.
\par
Let $m$ be the $\sigma$-ergodic measure in $\W$ appearing in Hypotheses \ref{3.hipos},
and represent by $\alpha_\ep(m)$ the corresponding
rotation number of the family \eqref{3.hamilep}, so that $\alpha_0(m)=0$.
According to Theorem 5.2 of \cite{fajn1}
(see also Theorem 2.28 of \cite{jonnf}), the exponential dichotomy
forces $\alpha_\ep(m)$ to take values in a discrete group if $\ep\in[0,\ep_0]$.
On the other hand, Theorem 4.3 of \cite{fajn1} (or Theorem 2.25 of \cite{jonnf})
ensures that $\alpha_\ep(m)$ varies continuously with respect to $\ep$.
Since $\alpha_0(m)=0$, we conclude that $\alpha_\ep(m)=0$ for $\ep\in[0,\ep_0]$.
\par
Since $B_\ep(\w_0)$ is nonsingular for any $\ep>0$, the problems for
$\ep\in(0,\ep_0]$ fulfill the corresponding Hypotheses~\ref{3.hipos}.
Hence, Theorem~\ref{3.teorBregular}
ensures the uniform weak disconjugacy of the families \eqref{3.hamilep},
the global existence of the Weyl function $M^+_\ep\colon\W\to\s_n(\R)$
associated to \eqref{3.hamilep}, and the
solvability of the minimization problems for $\mI^\ep_{\bx_0,\w}$ subject to
\eqref{3.controlwep}. In addition,
given $\w\in\W$ and $\bx_0\in\R^n$, the pair
$(\wit\bx^\ep_\w(t),\wit\bv^\ep_\w(t))$ minimizing $\mI^\ep_{\bx_0,\w}$ comes
from the solution $\lsm\wit\bx^\ep_\w(t)\\\wit\by^\ep_\w(t)\rsm$
of \eqref{1.hamilw}$_\w^\ep$ with initial data
$\lsm \bx_0\\M_\ep^+(\w)\,\bx_0\rsm$ via the analogous of the
feedback rule \eqref{1.frulew}$_\w$,
and the value of the minimum is
\begin{equation}\label{3.minep}
 \mI^\ep_{\bx_0,\w}(\wit\bx^\ep_\w,\wit\bv^\ep_\w)=-\frac{1}{2}\:\bx_0^TM_\ep^+(\w)\,\bx_0\,.
\end{equation}
For further purposes we point out that if $\bv\colon\R\to\R^n$ is written
as $\bv(t)=\lsm\bu(t)\\\bv_2(t)\rsm$ with $\bu\colon\R\to\R^m$, then
\begin{equation}\label{3.relQ}
 2\,\mQ^\ep_\w(t,\bx(t),\bv(t))=2\,\mQ_\w(t,\bx(t),\bu(t))+\ep\,|\bv_2(t)|^2\,,
\end{equation}
as easily deduced from the respective definitions \eqref{1.defQw}$_\w$
and \eqref{3.defQwep}$_\w$ of $\mQ_\w$ and $\mQ_\w^\ep$.
\par
It follows from \eqref{3.Bep} that $H_\ep$ satisfies the condition of monotonicity
required by Proposition 5.51 of \cite{jonnf}.
This result combined with Theorem~\ref{3.teorBregular}(ii) ensures a monotone behaviour
of the Weyl functions. In particular,
$M_{\ep_2}^+\le M_{\ep_1}^+$ if $\ep_0>\ep_1>\ep_2>0$.
Therefore, we can apply Theorem 1 of Kratz \cite{krat0}, which
establishes two alternatives for the limiting behaviour of $\bx_0^T M_\ep^+(\w)\,\bx_0$
as $\ep\to 0^+$, which depend of the pair $(\bx_0,\w)\in\R^n\times\W$\,:
\begin{itemize}
\item[{\bf (a)}] $\lim_{\ep\to 0^+} \bx_0^T M_\ep^+(\w)\,\bx_0$ belongs to $\R$ if and
only if there exists $\by_0\in\R^n$ such that $\lsm\bx_0\\\by_0\rsm\in l^+(\w)$. In this case,
if $l^+(\w)\equiv \lsm L_{\w,1}\\L_{\w,2}\rsm$ and $\lsm\bx_0\\\by_0\rsm=\lsm
L_{\w,1}\,\bc\\L_{\w,2}\,\bc\rsm$ for a vector $\bc\in\R^n$, then
$\lim_{\ep\to 0^+} \bx_0^T M_\ep^+(\w)\,\bx_0=\bc^TL^T_{\w,2}\,L_{\w,1}\,\bc$.
\item[{\bf (b)}] $\lim_{\ep\to 0^+} \bx_0^T M_\ep^+(\w)\,\bx_0=-\infty$ otherwise.
\end{itemize}
We will prove that
\begin{itemize}
\item[{\bf 1.}] if the pair $(\bx_0,\w)$ is in case {\bf (a)}, then there exist
admissible pairs for $\mI_{\bx_0,\w}$, and all the assertions in Theorem~\ref{1.teormain}(ii)
are true.
\item[{\bf 2.}] If the pair $(\bx_0,\w)$ is in case {\bf (b)}, then
$\mI_{\bx_0,\w}(\bx,\bu)=\infty$ for all $(\bx,\bu)\in\mP_{\bx_0,\w}$, so that
there are no admissible pairs.
\end{itemize}
These two facts will complete the proof in the particular case we have started with.
\par
Let us assume that $(\bx_0,\w)$ is in the situation {\bf (a)}, and take $\bc\in\R^n$
as there. We define $\lsm L_1(t)\\L_2(t)\rsm$ as the $2n\times n$ matrix valued
solution of \eqref{1.hamilw} with $\lsm L_1(0)\\L_2(0)\rsm=\lsm L_{\w,1}\\L_{\w,2}\rsm$.
We also define
\[
\begin{split}
 \wit\bx_\w(t)&:=L_1(t)\,\bc\,,\quad\wit\by_\w(t):=L_2(t)\,\bc\,,\\
 \wit\bu_\w(t)&:=R^{-1}(\wt)\,B^T(\wt)\,\wit\by_\w(t)-R^{-1}(\wt)\,g^T(\wt)\,\wit\bx_\w(t)\,,\\
 V(t)&:=\wit\by_\w^T(t)\,\wit\bx_\w(t)\,.
\end{split}
\]
Note that the pair $(\wit\bx_\w,\wit\bu_\w)$ belongs to
$\mP_{\bx_0,\w}$ and is the one of Theorem \ref{1.teormain}(ii).
It follows from Definition \ref{2.defED}, from the fact that
$\lsm\wit\bx_\w(0)\\\wit\by_\w(0)\rsm=\lsm\bx_0\\\by_0\rsm\in l^+(\w)$, and from
the definition of $\wit\bu_\w$ that
$(\wit\bx_\w,\wit\bu_\w)\in L^2(\R^+,\R^n)\times L^2(\R^+,\R^m)$; thus,
$(\wit\bx_\w,\wit\bu_\w)$ is an admissible pair for $\mI_{\bx_0,\w}$.
Lemma \ref{3.lemaV} yields
\[
 \mI_{\bx_0,\w}(\wit\bx_\w,\wit\bu_\w)=
 \frac{1}{2}\left(\lim_{t\to\infty} \wit\by_\w^T(t)\,\wit\bx_\w(t)-
  \wit\by_\w^T(0)\,\wit\bx_\w(0)\right)
  =-\frac{1}{2}\:\bc^TL^T_{\w,2}\,L_{\w,1}\,\bc\,.
\]
Here we use that $\lim_{t\to\infty}\lsm \wit\bx_\w(t)\\\wit\by_\w(t)\rsm=\lsm\bcero\\\bcero\rsm$,
which also follows from $\lsm\bx_0\\\by_0\rsm\in l^+(\w)$.
\par
Our next step is proving that $\mI_{\bx_0,\w}(\bar\bx,\bar\bu)\ge
-(1/2)\,\bc^TL^T_{\w,2}\,L_{\w,1}\,\bc$
for any other admissible pair $(\bar\bx,\bar\bu)$. Given such a pair,
we define $\bar\bv\colon\R\to\R^n$ by $\bar\bv(t)=\lsm\bar\bu(t)\\\bcero\rsm$.
Since $B(\wt)\,\bar\bu(t)=B_\ep(\wt)\,\bar\bv(t)$, the pair
$(\bar\bx,\bar\bv)$ is admissible for $\mI^\ep_{\bx_0,\w}$ for all $\ep>0$.
It follows from \eqref{3.relQ} and \eqref{3.minep} that
\begin{equation}\label{3.cota}
 \mI_{\bx_0,\w}(\bar\bx,\bar\bu)=\mI^\ep_{\bx_0,\w}(\bar\bx,\bar\bv)\ge
 \mI^\ep_{\bx_0,\w}(\wit\bx_\w,\wit\bv_\w)=-\frac{1}{2}\:\bx_0^TM^+_\ep(\w)\,\bx_0\,,
\end{equation}
so that the assertion follows from the information provided by {\bf (a)}
by taking limit as $\ep\to 0^+$. This completes the proof of {\bf 1}.
\par
In order to check {\bf 2}, we assume that
the pair $(\bx_0,\w)$ is in case {\bf (b)} and assume for contradiction that
there exists an admissible pair $(\bar\bx,\bar\bu)$ for
$\mI_{\bx_0,\w}$. Repeating the previous procedure we obtain
\eqref{3.cota}. Taking limit as $\ep\to 0^+$ and using the information
provided by {\bf(b)} we conclude that $\mI_{\bx_0,\w}(\bar\bx,\bar\bu)=\infty$,
which precludes the admissibility of the pair. This
completes the proofs of point {\bf 2} and of the initial case.
\smallskip\par
{\sc General case}. Basic results on linear algebra provide an orthogonal
$n\times n$ matrix valued function $P$ such that $P\,B(\w_0)=\lsm B_1(\w_0)\\B_2(\w_2)\rsm$
for an $m\times m$ nonsingular matrix $B_1(\w_0)$. Let us define
\[
\begin{split}
 \wit A(\w)&:= P\,A(\w)\,P^T\,,\\
 \wit B(\w)&:= P\,B(\w)\,,\\
 \wit G(\w)&:= P\,G(\w)\,P^T\,,\\
 \wit g(\w)&:= P\,g(\w)\,,
\end{split}
\]
and consider the families of control systems and functionals
\begin{align}
 \bz'=&\,\wit A(\wt)\,\bz+\wit B(\wt)\,\bu\,,\label{3.controlABt}\\
 \wit\mQ_\w(t,\bz,\bu):=&\,\Frac{1}{2}\left(\langle\,\bz,\wit G(\wt)\,\bz\,\rangle
 +2\langle\,\bz,\wit g(\wt)\,\bu\,\rangle+\langle\,\bu, R(\wt)\,\bu\,\rangle\right),\label{3.defQwt}\\
 \wit \mI_{\bz_0,\w}(\bz,\bu):=&\,
 \int_0^{\infty}\wit \mQ_\w(t,\bz(t),\bu(t))\,dt\,,\label{3.defIwt}
\end{align}
obtained from the initial ones by means of the change of variables $\bz(t)=P\,\bx(t)$.
The corresponding family of linear hamiltonian systems is
\begin{equation}\label{3.hamilwt}
 \left[\!\begin{array}{c}\bz\\\bw\end{array}\!\right]'=
 \wit H(\wt)\,\left[\!\begin{array}{c}\bz\\\bw\end{array}\!\right],\qquad \w\in\W
\end{equation}
with
\[
 \wit H:=\left[\begin{array}{cc} \wit A-\wit B\,R^{-1}\wit g^T&\;
                        \wit B\,R^{-1}\wit B^T\\
                   \wit G-\wit g\,R^{-1}\wit g^T &\; -\wit A^T\!+\wit g\,R^{-1}\wit B^T
                   \end{array}\right].
\]
A straightforward computation shows that \eqref{3.hamilwt} comes from \eqref{1.hamilw}
by means of the change of variables
\[
 \left[\!\begin{array}{c}\bz\\\bw\end{array}\!\right]=
 \left[\begin{array}{cc}P&0_n\\0_n&P\end{array}\right]
 \left[\!\begin{array}{c}\bx\\\by\end{array}\!\right]\,.
\]
It is clear that the family \eqref{3.hamilwt}
has exponential dichotomy over $\W$, since the change of variables
is given by a constant matrix, and that
$\lsm P&0_n\\0_n&P\rsm\lsm L_{\w,1}\\L_{\w,2}\rsm$
represents the Lagrange plane $\wit l^+(\w)$ of the solutions of \eqref{3.hamilwt}
bounded at $+\infty$ if and only if $\lsm L_{\w,1}\\L_{\w,2}\rsm$
represents $l^+(\w)$. In addition, since the matrix
$\lsm P&0_n\\0_n&P\rsm$ is symplectic, it follows from the results
of Section 2 of \cite{nono} (see also Section 2.1.1 of \cite{jonnf})
that the rotation number is also preserved: it is 0 for any $\sigma$-ergodic measure.
Therefore the transformed families are in the situation analyzed in the particular case.
It is clear that:
\begin{itemize}
\item[--] the pair $(\bar\bx,\bar\bu)$ is admissible for $\mI_{\bx_0,\w}$
if and only if the pair $(\bar\bz,\bar\bu)$ given by
$\bz(t)=P\,\bx(t)$ is admissible for $\wit\mI_{P\bx_0,\w}$.
\item[--] There exists $\by_0$ such that $\lsm\bx_0\\\by_0\rsm\in l^+(\w)$ if and only if
there exists $\bw_0$ such that $\lsm P\,\bx_0\\\bw_0\rsm\in\wit l^+(\w)$: just write
$\bw_0=P\by_0$.
\item[--] $\bc^T L_{\w,2}^TL_{\w,1}\,\bc=\bc^T \big(PL_{\w,2}\big)^T\big(P\,L_{\w,1}\big)\,\bc$\,.
\item[--] The equality $\bu(t)= R^{-1}(\wt)\,B^T(\wt)\,\by(t)- R^{-1}(\wt)\,g^T(\wt)\,\bx(t)$
holds if and only if
$\bu(t)= R^{-1}(\wt)\,\wit B^T(\wt)\,\bw(t)- R^{-1}(\wt)\,\wit g^T(\wt)\,\bz(t)$
for $\bz(t)=P\,\bx(t)$ and $\bw(t)=P\,\by(t)$.
\end{itemize}
It is easy to deduce the statements of Theorem~\ref{1.teormain} from
all these properties. The proof is hence complete. \hfill{\qed}
\begin{nota}\label{3.notaunico}
Let us represent $l^+(\w)\equiv\lsm L_{\w,1}\\L_{\w,2}\rsm$
and assume that $L_{\w,1}\bc=L_{\w,1}\bd$
for $\bd\ne\bc$. Then, since $l^+(\w)$ is a Lagrange plane,
$\bc^TL_{\w,2}^TL_{\w,1}\bc=\bd^TL_{\w,2}^TL_{\w,1}\bd=\bd^TL_{\w,2}^TL_{\w,1}\bc$.
We point out this property to clarify that there is not ambiguity
in the assertion of Theorem~\ref{1.teormain} concerning the value of the
minimum of $\mI_{\bx_0,\w}$, which of course is unique.
\end{nota}
\subsection{Coming back to a single problem}
Recall that our starting point was the single problem given by
\eqref{1.control}, \eqref{1.defQ} and \eqref{1.defI}, which we included in the family
given by \eqref{1.controlw}, \eqref{1.defQw} and \eqref{1.defIw} by means of the
hull procedure explained in Section \ref{2.sechull}. The initial problem is
one of those of the family: it corresponds to the point $\wit\w=C_0=(A_0,B_0,G_0,g_0,R_0)$
of $\W$.
\par
The conclusions of Theorem~\ref{1.teormain} apply to every $\w\in\W$, so that
they also apply to our initial problem. What about the hypotheses?
\begin{itemize}
\item[--] The hull $\W$ is minimal in the cases of recurrence
of the initial coefficients,
which includes (at least) the autonomous, periodic, quasi-periodic, almost periodic and
almost automorphic cases.
\item[--] The hypothesis on $B$ holds if there exists
$t_0\in\R$ such that $B_0(t_0)$ has full rank.
\item[--] In the case of minimality of $\W$, the exponential dichotomy of the family
\eqref{1.hamilw} holds if and only if the initial Hamiltonian system \eqref{1.hamil}
has exponential dichotomy on $\R$: see Remark \ref{2.notasED}.1.
\item[--] And if the base is minimal and uniquely ergodic (which is the case at least
if $C_0$ is almost periodic: see \cite{fink}), then the value of
the rotation number can be obtained
for any one of the systems of the family; for instance, for the initial one
(see Theorem 2.6 of \cite{jonnf}).
\end{itemize}
Therefore, a less general formulation of our main theorem reads as follows:
\begin{teor}\label{3.teorini}
Assume that the map $C_0$ given by \eqref{2.defC0} is almost periodic with
$R_0(t)\ge\rho\,I_m$ for a common $\rho>0$ and any $t\in\R$, that
there exists $t_0\in\R$ such that the $n\times m$ matrix $B_0(t_0)$ has full rank,
that the Hamiltonian system \eqref{1.hamil} has exponential dichotomy on $\R$, and that
\[
 \lim_{t\to\infty}\frac{1}{t}\:\arg\det (U_1(t)-iU_2(t))=0\,,
\]
where $U_H=\lsm U_1&U_3\\U_2&U_4\rsm$ is the matrix solution of \eqref{1.hamil}
with $U_H(0)=I_{2n}$.
Let $l^+$ be the Lagrange plane of the solutions of \eqref{1.hamil} bounded at $+\infty$,
and let us fix $\bx_0\in\R^n$. Then,
\begin{itemize}
\item[\rm(i)] there exist admissible pairs $(\bx,\bu)$ for the functional
$\mI_{\bx_0}$ given by \eqref{1.defI} if and only if there exists $\by_0\in\R^n$ such that
$\lsm\bx_0\\\by_0\rsm\in l^+$.
\item[\rm(ii)] In this case, the infimum of $\mI_{\bx_0}$ is finite.
In addition, if the columns of the
$2n\times n$ matrix $\lsm L_1\\L_2\rsm$ are a basis of $l^+$ and
$\lsm\bx_0\\\by_0\rsm=\lsm L_1\,\bc\\L_2\,\bc\rsm$ for a vector $\bc\in\R^n$, then
the infimum is given by $-(1/2)\,\bc^T L^T_2\,L_1\,\bc$,
and a minimizing pair $(\wit\bx,\wit\bu)$ is obtained from the solution
$\lsm\wit\bx(t)\\\wit\by(t)\rsm$ of \eqref{1.hamil} with initial data
$\lsm\bx_0\\\by_0\rsm$ via the feedback rule \eqref{1.frule}.
\end{itemize}
\end{teor}
Note also that the condition of almost periodicity of $C_0$ can be replaced by the
less restrictive one of minimality and ergodic uniqueness of the flow on its hull.
\subsection{Examples of non global solvability}
We complete the paper with three examples. With the first one, of autonomous type,
we intend to show a simple scenario, in which every required computation
can be done by hand, and for which the existence
of admissible pairs depends on the choice of the initial data. The second one
is a generalization of almost periodic type.
\par
The third example, more
complex, shows a situation of non global solvability for which the associated
linear Hamiltonian system does not have exponential dichotomy,
so that one of the hypotheses of our results is not fulfilled.
In fact, the Hamiltonian system of this example is of nonuniform
hyperbolic type. Its aim is showing that the same ideas involved in the
description we have made can be extremely useful
in the analysis of other situations.
\begin{exa}\label{3.exaaut}
We consider the autonomous control problem and quadratic functional
\[
\begin{split}
 \left[\!\begin{array}{c}x_1\\x_2\end{array}\!\right]'=&\,
 \left[\begin{array}{cc} 1&1\\0&1\end{array}\right]\left[\!\begin{array}{c}x_1\\x_2\end{array}\!\right]+
 \left[\begin{array}{c} 1\\0\end{array}\right]u\,,\\
 \mQ\big(t,\lsm x_1\\x_2\rsm,u\big):=&\,\,\frac{1}{2}\,\left([\,x_1\;x_2\,]\left[\begin{array}{cc} 2&1\\1&1\end{array}\right] \left[\!\begin{array}{c}x_1\\x_2\end{array}\!\right]+
 2\,\,[\,x_1\;x_2\,]\left[\begin{array}{c}1\\1\end{array}\right]u+u^2\right),
\end{split}
\]
and we pose the problem of minimizing the corresponding functional
\begin{equation}\label{3.defIex}
 \mI_{\bx_0}\colon\mP_{\bx_0}\to\R\cup\{\pm\infty\}\,,
 \quad \big(\lsm x_1\\x_2\rsm,u\big)\mapsto\int_0^\infty
 \mQ\Big(t,\lsm x_1(t)\\x_2(t)\rsm,u(t)\Big)\,dt
\end{equation}
defined on the set $\mP_{\bx_0}$ of the measurable pairs $(\bx,u)$
satisfying the control system with $\bx(0)=\bx_0$.
\par
Let us check that all the hypotheses of Theorem~\ref{3.teorini} are satisfied.
It is obvious that the map $C_0$ given by the expression \eqref{2.defC0}
corresponding to this problem is constant
(and hence almost periodic), and that the rank of $B_0(t)\equiv\lsm 1\\0\rsm$ is
$1$ for all $t\in\R$. It is also easy to check that the linear
Hamiltonian system \eqref{1.hamil} takes the form
\[
 \left[\!\begin{array}{c}x_1\\x_2\\y_1\\y_2\end{array}\!\right]'=
 H\,\left[\!\begin{array}{c}x_1\\x_2\\y_1\\y_2\end{array}\!\right]\quad
 \text{for\;\,}
 H:=\left[\begin{array}{cccr}0&0&1&0\\0&1&0&0\\1&0&0&0\\0&0&0&-1\end{array}\right].
\]
Note that this system can be uncoupled to
\[
 \left[\begin{array}{c}x_1\\y_1\end{array}\right]'
 =\left[\begin{array}{cc} 0&1\\1&0\end{array}\right]
 \left[\begin{array}{c}x_1\\y_1\end{array}\right]
 \quad\text{and}\quad
 \left[\begin{array}{c}x_2\\y_2\end{array}\right]'
 =\left[\begin{array}{rc} \!1&\;\,0\\0&-1\end{array}\right]
 \left[\begin{array}{c}x_2\\y_2\end{array}\right].
\]
It is very simple to check that this two-dimensional systems (of Hamiltonian type)
have exponential dichotomy in $\R$.
In addition, the initial data of the solutions bounded at $+\infty$ and $-\infty$ are
given for the first one by $\lsm\;\,\,1\\-1\rsm$ and $\lsm 1\\1\rsm$ (up to constant multiples)
and for the second one by
$\lsm 0\\1\rsm$ and $\lsm 1\\0\rsm$.
It follows that our four-dimensional Hamiltonian system has exponential dichotomy in
$\R$ (see e.g.~Remark~\ref{2.notasED}.1), and that the Lagrange planes
\[
 l^+\equiv \left[\!\begin{array}{rr} 1&0\\0&0\\
 -1&0\\0&\;1\end{array}\right]
 \quad\text{and}\quad
 l^-\equiv \left[\begin{array}{cc} 1&0\\0&1\\
 1&0\\0&0\end{array}\right]
\]
are composed of the initial data of the solutions of the Hamiltonian system
which are bounded as $t\to+\infty$ and $t\to-\infty$, respectively.
It is also easy to compute fundamental matrix solutions of the two-dimensional systems,
and from them
\[
 U_H(t)=\frac{1}{2}\left[\begin{array}{cccc}e^t+e^{-t}&0&e^t-e^{-t}&0 \\
  0 &2\,e^t&0&0 \\e^t-e^{-t}&0&e^t+e^{-t}&0\\0&0&0&2\,e^{-t}\end{array}\right]\,,
\]
which is the matrix solution the Hamiltonian system with $U_H(0)=I_4$.
It follows from here that the last hypothesis of Theorem~\ref{3.teorini}
holds, since
\[
\begin{split}
 &\lim_{t\to\infty}\frac{1}{t}\arg\det\frac{1}{2}
 \left(\left[\begin{array}{cc}e^t+e^{-t}&0\\0&2\,e^t
 \end{array}\right]
 -i\left[\begin{array}{cc} e^t-e^{-t}&0\\0&0\end{array}\right]\right)\\
 &\quad\qquad\qquad\qquad\qquad\qquad\qquad
 =\lim_{t\to\infty}\frac{1}{t}\arctan \frac{1-e^{2t}}{1+e^{2t}}=0\,.
\end{split}
\]
\par
Theorem~\ref{3.teorini} ensures that there exist admissible pairs if
and only if $\bx_0=\lsm x_1\\x_2\rsm=\lsm 1&0\\0&0\rsm\,\lsm c_1\\c_2\rsm$ for
some $\bx=\lsm c_1\\c_2\rsm\in\R^2$. That is, if and only if
$x_2=0$, in which case we can take $\bc=\lsm x_1\\c_2\rsm$
for any $c_2\in\R$. This provides $\by_0=\lsm-1&0\\\;\,\,0&1\rsm\lsm x_1\\c_2\rsm=
\lsm -x_1\\\;\,\,c_2\rsm$.
In addition, also according to Theorem~\ref{3.teorini}, the minimum is given by
$-(1/2)[\,x_1\;c_2\,]\lsm -1&0\\\;\,\,0&1\rsm\lsm 1&0\\0&0\rsm\lsm x_1\\c_2\rsm=
x_1^2/2$. Now we compute
\[
 \left[\,\begin{array}{r}x_1(t)\\x_2(t)\\y_1(t)\\y_2(t)\end{array}\!\right]=
 U_H(t)\left[\,\!\begin{array}{r}x_1\\0\,\,\\-x_1\\c_2\end{array}\!\right]=
 \left[\,\!\begin{array}{r}x_1\,e^{-t}\\0\;\quad\\-x_1\,e^{-t}\\c_2\,e^{-t}\end{array}\!\right],
\]
apply the feedback rule \eqref{1.frule} in order to obtain the control
\[
 \wit u(t)=
 [\,1\;0\,]\,\left[\,\!\begin{array}{r}-x_1\,e^{-t}\\c_2\,e^{-t}\end{array}\!\right]-
 [\,1\;1\,]\,\left[\,\!\begin{array}{c}x_1\,e^{-t}\\0\end{array}\!\right]=-2\,x_1\,e^{-t}\,,
\]
and conclude that there is a unique minimizing pair
$(\wit\bx,\wit u)$ with the shape described in Theorem \ref{3.teorini},
given by
\[
 \wit\bx(t)=\left[\,\!\begin{array}{c}x_1\,e^{-t}\\0\end{array}\!\right],\qquad
 \wit u(t)=-2\,x_1\,e^{-t}\,.
\]
\end{exa}
\begin{exa}\label{3.exa2}
Let $f\colon\R\to\R$ be an almost periodic function such that
\begin{equation}\label{3.sisalmost}
 \left[\!\begin{array}{c}x_1\\y_1\end{array}\!\!\right]'
 =\left[\begin{array}{cccc} 0&1\\f(t)&0\end{array}\right]
 \left[\!\begin{array}{c}x_1\\y_1\end{array}\!\!\right]
\end{equation}
satisfies two conditions. The first one is the existence of
exponential dichotomy on $\R$ for \eqref{3.sisalmost}
combined with the fact that the initial data
of the solutions bounded at $+\infty$ and $-\infty$ are multiples of
$\lsm 1\\m^+\rsm$ and $\lsm 1\\m^-\rsm$. The second one is that
the rotation number (with respect to the unique ergodic measure
which exists in the hull in the almost periodic case) is zero.
As explained in the comments before Theorem~\ref{3.teorini},
this fact is is equivalent to say that, if $V(t)=\lsm V_1(t)&V_3(t)\\V_2(t)&V_4(t)\rsm$
is the matrix solution of \eqref{3.sisalmost} with $V(0)=I_2$, then
\begin{equation}\label{3.rot0}
 \lim_{t\to\infty}\frac{1}{t}\:\arg(V_1(t)-iV_2(t))=0\,.
\end{equation}
The function $f$ can be very simple. For instance, $f\equiv 1$, as in the previous example.
But it can also be extremely complex. For instance,
$f:=f_0+\lambda$ for any $\lambda>0$ where $f_0$
is the almost periodic function described by Johnson in~\cite{john12}, giving rise
to a non uniformly hyperbolic family of Schr\"{o}dinger equations: as
explained in Example 7.37 of \cite{jonnf}, the corresponding system \eqref{3.sisalmost}
has exponential dichotomy (if $\lambda>0$), the initial data of the solutions
bounded at $+\infty$ and $-\infty$ are $\lsm 1\\m^+\rsm$ and $\lsm 1\\m^-\rsm$
(up to a multiple), and the solution
$\lsm x_1(t)\\x_2(t)\rsm=V(t)\lsm 1\\m^+\rsm$ satisfies $x_1(t)\ne 0$ for all $t\in\R$,
from where we deduce that the rotation number is zero (use for instance
Propositions 5.8 and 5.65 of \cite{jonnf}).
We will refer again to the function $f_0$ in
Example \ref{3.exa3}.
\par
Let $h\colon\R\to\R$ be any other almost periodic function.
Given the control problem and the quadratic functional
\[
\begin{split}
 \left[\!\begin{array}{c}x_1\\x_2\end{array}\!\right]'=&\,\left[\begin{array}{cc} 1&h(t)\\0&1\end{array}\right]
 \left[\!\begin{array}{c}x_1\\x_2\end{array}\!\right]+
 \left[\begin{array}{c} 1\\0\end{array}\right]u\,,\\
 \mQ(t,\lsm x_1\\x_2\rsm,u):=&\;[\,x_1\;x_2\,]\left[\begin{array}{cl} f(t)+1&h(t)\\h(t)&h(t)^2\end{array}\right]\left[\!\begin{array}{c}x_1\\x_2\end{array}\!\right]+
 [\,x_1\;x_2\,]\left[\!\begin{array}{c}1\\h(t)\end{array}\! \right]u+u^2
\end{split}
\]
we pose the problem of minimizing the corresponding functional
\eqref{3.defIex}.
As in the previous example, we will begin by checking that all the
hypotheses of Theorem~\ref{3.teorini} are satisfied.
It is obvious that the map $C_0$ given by \eqref{2.defC0} is almost periodic, and
the rank of $B_0(t)\equiv\lsm 1\\0\rsm$ is $1$ for all $t\in\R$. Now, the
Hamiltonian system \eqref{1.hamil} takes the form
\[
 \left[\!\begin{array}{c}x_1\\x_2\\y_1\\y_2\end{array}\!\right]'=
 \left[\begin{array}{cccr}0&0&1&0\\0&1&0&0\\f(t)&0&0&0\\0&0&0&-1\end{array}\right]
 \left[\!\begin{array}{c}x_1\\x_2\\y_1\\y_2\end{array}\!\right]\,,
\]
which can be uncoupled to
\[
 \left[\begin{array}{c}x_1\\y_1\end{array}\right]'
 =\left[\begin{array}{cr} 0&1\\f(t)&0\end{array}\right]
 \left[\begin{array}{c}x_1\\y_1\end{array}\right]
 \quad\text{and}\quad
 \left[\begin{array}{c}x_2\\y_2\end{array}\right]'
 =\left[\begin{array}{rc} \!1&\;\,0\\0&-1\end{array}\right]
 \left[\begin{array}{c}x_2\\y_2\end{array}\right].
\]
As seen in Example~\ref{3.exaaut}, the second system has also exponential dichotomy,
and $\lsm 0\\1\rsm$ and $\lsm 1\\0\rsm$ are
the initial data of the solutions bounded at $+\infty$ and $-\infty$.
Hence, our four-dimensional Hamiltonian system has exponential dichotomy in
$\R$, and the Lagrange planes
\[
 l^+\equiv \left[\!\begin{array}{cc} 1&0\\0&0\\
 \,m^+&0\\0&\;1\end{array}\right]
 \quad\text{and}\quad
 l^-\equiv \left[\begin{array}{cc} 1&0\\0&1\\
 \,m^-&0\\0&0\end{array}\right]
\]
are composed of the initial data of the solutions of the four-dimensional system
which are bounded as $t\to+\infty$ and $t\to-\infty$, respectively.
In addition, the matrix solution $U_H(t)$ with $U_H(0)=I_4$ is given by
\[
 U_H(t)=\left[\begin{array}{cccc}V_1(t)&0&V_3(t)&0 \\
  0 &e^t&0&0 \\V_2(t)&0&V_4(t)&0\\0&0&0&e^{-t}\end{array}\right]\,,
\]
and, using \eqref{3.rot0}, we see that
\[
\begin{split}
 &\lim_{t\to\infty}\frac{1}{t}\arg\det\left(\left[\begin{array}{cc}V_1(t)&0\\0&e^t
 \end{array}\right]
 -i\left[\begin{array}{cc} V_2(t)&0\\0&0\end{array}\right]\right)\\
 &\qquad\qquad\qquad\qquad=\lim_{t\to\infty}\frac{1}{t}\:
 \arg\big(e^t(V_1(t)-iV_2(t))\big)=0\,.
\end{split}
\]
Therefore, all the hypotheses of Theorem~\ref{3.teorini} are satisfied, as asserted.
As in Example \ref{3.exaaut}, we conclude that there exist admissible pairs if
and only if $\bx_0=\lsm x_1\\0\rsm=\lsm 1&0\\0&0\rsm\,\lsm x_1\\c_2\rsm$
for any $c_2\in\R$, which provides $\by_0=\lsm m^+&0\\0&1\rsm\lsm x_1\\c_2\rsm=
\lsm m^+\!x_1\\c_2\rsm$. In this case the minimum of the functional
is $-x_1^2\,m^+\!/2\,$. And, since
\[
 \left[\begin{array}{cccc}V_1(t)&0&V_3(t)&0 \\
  0 &e^t&0&0 \\V_2(t)&0&V_4(t)&0\\0&0&0&e^{-t}\end{array}\right]
 \left[\!\begin{array}{c}x_1\\0\\m^+ x_1\\c_2\end{array}\!\right]=
 \left[\!\begin{array}{c}(V_1(t)+V_3(t)\,m^+)\,x_1\\0\\(V_2(t)+V_4(t)\,m^+)\,x_1\\c_2\,e^{-t}
 \end{array}\!\right],
\]
then the feedback rule \eqref{1.frule}
provides the minimizing pair $(\wit\bx,\wit u)$, with
\[
\wit\bx(t)=\left[\!\begin{array}{c} (V_1(t)+V_3(t)\,m^+)\,x_1\\0\end{array}\!\right],
\quad \wit u(t)=\big(V_2(t)-V_1(t)+(V_4(t)-V_3(t))\,m^+\big)\,x_1\,.
\]
\end{exa}
\begin{exa}\label{3.exa3}
Let $f_0$ be the almost periodic function described by Johnson in~\cite{john12},
already mentioned in Example \ref{3.exa2}, and let $h_0$ be any almost periodic
function with frequency modulus contained in that of $f_0$. Our aim is to analyze the
minimization problem for the functional
\[
 \mI_{\bx_0}\colon\mP_{\bx_0}\to\R\cup\{\pm\infty\}\,,\quad
 \big(\lsm x_1\\x_2\rsm,u\big)\mapsto\int_0^\infty
 \mQ\Big(t,\lsm x_1(t)\\x_2(t)\rsm,u(t)\Big)\,dt,
\]
which is evaluated on the set $\mP_{\bx_0}$ of measurable pairs $(\bx,u)$ solving
\[
 \left[\!\begin{array}{c}x_1\\x_2\end{array}\!\right]'=\left[\begin{array}{cc} 1&h_0(t)\\0&1\end{array}\right]
 \left[\!\begin{array}{c}x_1\\x_2\end{array}\!\right]+
 \left[\begin{array}{c} 1\\0\end{array}\right]u
\]
with $\bx(0)=\bx_0$, and which is given by
\[
 \mQ(t,\lsm x_1\\x_2\rsm,u):=[\,x_1\;x_2\,]\left[\begin{array}{cl} f_0(t)+1&h_0(t)\\h_0(t)&h_0(t)^2\end{array}\right]
 \left[\!\begin{array}{c}x_1\\x_2\end{array}\!\right]+
 [\,x_1\;x_2\,]\left[\!\begin{array}{c}1\\h_0(t)\end{array}\! \right]u+u^2.
\]
For reasons which will became clear later, we must
consider in this case the (common) hull $\W$ of $f_0$ and $h_0$,
whose construction we explained in Section \ref{2.sechull}.
This provides the families
\begin{align*}
\left[\!\begin{array}{c}x_1\\x_2\end{array}\!\right]'=&\,\left[\begin{array}{cc} 1&
h(\wt)\\0&1\end{array}\right]
\left[\!\begin{array}{c}x_1\\x_2\end{array}\!\right]+
 \left[\begin{array}{c} 1\\0\end{array}\right]u\,,\\
 \mQ_\w(t,\lsm x_1\\x_2\rsm,u):=&\;[\,x_1\;x_2\,]\left[\begin{array}{cl}
 f(\wt)+1&\!h(\wt)\\h(\wt)&\!h(\wt)^2\end{array}\right]
 \left[\!\!\begin{array}{c}x_1\\x_2\end{array}\!\!\right]+
 [\,x_1\;x_2\,]\left[\!\!\begin{array}{c}1\\h(\wt)\end{array}\!\!\right]u+u^2,\\
 \mI_{\bx_0,\w}\colon\mP_{\bx_0,\w}&\to\R\cup\{\pm\infty\}\,,\quad
 \big(\lsm x_1\\x_2\rsm,u\big)\mapsto\int_0^\infty
 \mQ_\w\Big(t,\lsm x_1(t)\\x_2(t)\rsm,u(t)\Big)\,dt
\end{align*}
for $\w\in\W$, where $\mP_{\bx_0,\w}$ is the set
of the measurable pairs solving the control problem corresponding to $\w$ with
initial state $\bx_0$.
\par
The corresponding family \eqref{3.hamilwt} takes the form
\begin{equation}\label{3.hamilex}
 \left[\!\begin{array}{c}x_1\\x_2\\y_1\\y_2\end{array}\!\right]'=
 \left[\begin{array}{cccr}0&0&1&0\\0&1&0&0\\f(\wt)&0&0&0\\0&0&0&-1\end{array}\right]
 \left[\!\begin{array}{c}x_1\\x_2\\y_1\\y_2\end{array}\!\right]\,,
\end{equation}
which can be uncoupled to
\begin{equation}\label{3.unc}
 \left[\begin{array}{c}x_1\\y_1\end{array}\right]'
 =\left[\begin{array}{cr} 0&1\\f(\wt)&0\end{array}\right]
 \left[\begin{array}{c}x_1\\y_1\end{array}\right]
 \quad\text{and}\quad
 \left[\begin{array}{c}x_2\\y_2\end{array}\right]'
 =\left[\begin{array}{rc} \!1&\;\,0\\0&-1\end{array}\right]
 \left[\begin{array}{c}x_2\\y_2\end{array}\right].
\end{equation}
As said in Example \ref{3.exa2}, it is proved in \cite{john12}
that the left family of systems in \eqref{3.unc} does not have exponential
dichotomy over $\W$. According to Remark \ref{2.notasED}.1,
this assertion is equivalent to the the existence of at least a nontrivial bounded solution
for at least one $\w\in\W$. Clearly this bounded solution provides a bounded
solution for \eqref{3.hamilex}$_\w$, so that the four-dimensional family
does not have exponential dichotomy over $\W$. Hence, we cannot apply Theorem \ref{1.teormain}.
\par
Let us now take $\lambda>0$ and consider the new families
\begin{align*}
 \left[\!\begin{array}{c}x_1\\x_2\end{array}\!\right]'=&\,\left[\begin{array}{cc} 1&h(\wt)\\0&1\end{array}\right]
 \left[\!\begin{array}{c}x_1\\x_2\end{array}\!\right]+
 \left[\begin{array}{cc} 1&0\\0&\lambda\end{array}\right]\left[\!\begin{array}{c}v_1\\v_2\end{array}\!\right]\,,\\
 \mQ_\w^\lambda(t,\lsm x_1\\x_2\rsm,\lsm v_1\\v_2\rsm):=
 &\;[\,x_1\;x_2\,]\left[\begin{array}{cl}
 f(\wt)+\lambda+1&h(\wt)\\h(\wt)&
 h(\wt)^2\end{array}\right]\left[\!\begin{array}{c}x_1\\x_2\end{array}\!\right]\\
 &\quad + [\,x_1\;x_2\,]\left[\begin{array}{cc}1&0\\h(t)&0\end{array}\! \right]
 \left[\!\begin{array}{c}v_1\\v_2\end{array}\!\right]+
 [\,v_1\;v_2\,]\left[\begin{array}{cc} 1&0\\0&\lambda\end{array}\right]\left[\!\begin{array}{c}v_1\\v_2\end{array}\!\right],\\
 \mI^\lambda_{\bx_0,\w}\colon\mP^\lambda_{\bx_0,\w}\to&\;\R\cup\{\pm\infty\}\,,
 \quad\big(\lsm x_1\\x_2\rsm,\lsm v_1\\v_2\rsm\big)\mapsto\int_0^\infty
 \mQ^\lambda_\w\Big(t,\lsm x_1(t)\\x_2(t)\rsm,\lsm v_1(t)\\v_2(t)\rsm\Big)\,dt\,,
\end{align*}
with associated family of linear Hamiltonian systems given by
\begin{equation}\label{3.hamilexl}
 \left[\!\begin{array}{c}x_1\\x_2\\y_1\\y_2\end{array}\!\right]'=
 \left[\begin{array}{cccr}0&0&1&0\\0&1&0&\lambda\\f(\wt)+\lambda&0&0&0\\0&0&0&-1\end{array}\right]
 \left[\!\begin{array}{c}x_1\\x_2\\y_1\\y_2\end{array}\!\right]\,.
\end{equation}
Let us fix $\lambda>0$. Then the family of systems $\lsm x_1\\y_1\rsm'
=\lsm0&1\\f(\wt)+\lambda&0\rsm \lsm x_1\\y_1\rsm$
has exponential dichotomy over $\W$, and the Lagrange
planes of the solutions which are bounded at $\pm\infty$ are
represented by $\lsm 1\\m^\pm_\lambda(\w)\rsm$: see again Example 7.37 of \cite{jonnf}.
It is easy to check that also the family $\lsm x_2\\y_2\rsm'=\lsm 1&\;\lambda\\0&-1\rsm
\lsm x_2\\y_2\rsm$ has exponential dichotomy, with Lagrange
planes of the solutions bounded at $+\infty$ and $-\infty$
given by $\lsm 1\\-2/\lambda\rsm$ and $\lsm 1\\0\rsm$ respectively.
Therefore the family \eqref{3.hamilexl} has exponential dichotomy, with
\[
 l^+_\lambda(\w)\equiv \left[\!\begin{array}{cc} 1&0\\0&1\\
 \,m^+_\lambda(\w)&0\\0&\!-2/\lambda\end{array}\right]
 \equiv \left[\!\begin{array}{cc} 1&0\\0&\!-\lambda/2\\
 \,m^+_\lambda(\w)&0\\0&1\end{array}\right],\;\quad
 l^-_\lambda(\w)\equiv \left[\begin{array}{cc} 1&0\\0&1\\
 \,m^-_\lambda(\w)&0\\0&0\end{array}\right].
\]
As in Example \ref{3.exa2} we can check by direct computation
that the rotation number of \eqref{3.hamilexl} is zero.
In these conditions, Theorem \ref{3.teorBregular}
ensures the solvability of the problem for $\mI^\lambda_{\bx_0,\w}$
for all $(\bx_0,\w)$, being the value of the minimum (if $\bx_0=\lsm x_1\\ x_2\rsm$)
\begin{equation}\label{3.valmin}
 -\frac{1}{2}\;[\,x_1\;x_2\,]\left[\begin{array}{cc}m^+_\lambda(\w)&0
 \\0&-2/\lambda\end{array}\right]\left[\!\begin{array}{c}x_1\\x_2\end{array}\!\right]
 =-\frac{1}{2}\,x_1^2\,m_\lambda^+(\w)+\frac{1}{\lambda}\,x_2^2\,.
\end{equation}
\par
On the other hand, it follows from Theorem 5.58 and
Proposition 5.51 of \cite{jonnf}
that, if $0<\lambda_1\le\lambda_2$, then
$m^+_{\lambda_2}(\w)\le m^+_{\lambda_1}(\w)\le m^-_{\lambda_1}(\w)\le m^-_{\lambda_2}(\w)$.
Therefore, there exist the limits $n^\pm(\w):=\lim_{\lambda\to 0^+}m^\pm_\lambda(\w)$
for all $\w\in\W$.
(As a matter of fact, $n^\pm$ are the principal functions of the
system corresponding to $\lambda=0$, which is uniformly weakly disconjugate:
see for instance Theorem 5.61 of \cite{jonnf}).
Moreover, there exists an invariant subset
$\W_0\subsetneq\W$ with full measure (for the unique ergodic measure on
the hull) such that, if $\w\in\W_0$, then: $n^+(\w)<n^-(\w)$, and the
solution of \eqref{3.hamilex}$_\w$
with initial data $\lsm x_0\\n^\pm(\w)\,x_0\rsm$, which can be written as
$\lsm x_1(t)\\n^\pm(\wt)\,x_1(t)\rsm$, belongs to $L^2([0,\infty),\R^2)$
and satisfies
\begin{equation}\label{3.limex}
 \lim_{t\to\pm\infty}\left[\begin{array}{c} x_1(t)\\n^\pm(\wt)\,x_1(t)\end{array}\!\right]
 =\left[\begin{array}{c}0\\0\end{array}\!\right].
\end{equation}
The proofs of the last assertions follow from the properties
of this family of two-dimensional systems described in \cite{john12},
and are based in the fact that the functions $n^+$ and $n^-$ provide the
Oseledets subbundles associated to the negative and positive
Lyapunov exponents of the family.
The interested reader can find in Theorem 6.3(iii) of \cite{jnot} a
detailed proof (formulated for the quasiperiodic case, but
applicable without changes to the almost periodic case), and
in Section 8.7 of \cite{jonnf} an exhaustive description of
the construction of an example with similar
behavior.
\par
Let us now observe that the limits as $\lambda\to 0^+$ of
$l^\pm_\lambda(\w)$ are the Lagrange planes
\[
 l^+(\w)\equiv\left[\!\begin{array}{cc} 1&0\\0&0\\
 \,n^+(\w)&0\\0&1\end{array}\right]\quad\text{and}\quad
 l^-(\w)\equiv \left[\begin{array}{cc} 1&0\\0&1\\
 \,n^-(\w)&0\\0&0\end{array}\right].
\]
In addition, if $\w\in\W_0$, then the solutions of
\eqref{3.hamilex}$_\w$ with initial data in $l^\pm(\w)$
tend to $\bcero$ as $t\to\pm\infty$, as deduced
from \eqref{3.limex} and from the behavior of the
solutions of the right-hand system of \eqref{3.unc}.
We will see that, if $\w\in\W_0$, then the initial states
$\bx_0$ for which the minimization problem is solvable,
as well as the value of the minimum and a minimizing pair,
can be obtained from $l^+(\w)$. The analogy with the
situation described in Theorem \ref{1.teormain} will be obvious.
\par
Let us fix $\w\in\W_0$ and $\bx_0=\lsm x_1\\ x_2\rsm\in\R^2$, and let
$\big(\lsm \bar x_1\\\bar x_2\rsm,\bar u\big)$ be an admissible pair for
$\mI_{\bx_0,\w}$. It is easy to check that
$\mI_{\bx_0,\w}\big(\lsm \bar x_1\\\bar x_2\rsm,\bar u\big)=
\mI^\lambda_{\bx_0,\w}\big(\lsm \bar x_1\\\bar x_2\rsm,\lsm \bar u\\0\rsm\big)$
for any $\lambda>0$, which combined with \eqref{3.valmin}
ensures that
\begin{equation}\label{3.II}
 \mI_{\bx_0,\w}\big(\lsm \bar x_1\\\bar x_2\rsm,\bar u\big)\ge
 -\frac{1}{2}\,x_1^2\,m_\lambda^+(\w)+\frac{1}{\lambda}\,x_2^2
\end{equation}
for any $\lambda>0$. By taking limit as $\lambda\to 0^+$
we conclude from the admissibility and from
the existence of the real limit $n^+(\w):=\lim_{\lambda\to 0^+}m_\lambda^+(\w)$
that $x_2$ must be 0. And this is equivalent to
ensure that there exists $\lsm y_1\\y_2\rsm\in\R^2$ such that
$\lsm x_1\\x_2\\y_1\\y_2\rsm\in l^+(\w)$: we can take
$\lsm y_1\\y_2\rsm=\lsm n^+x_1\\c_2\rsm$
for any $c_2\in\R$.
\par
We work from now on with $\bx_0=\lsm x_1\\0\rsm$.
As just said, taking limit as $\lambda\to 0^+$ in \eqref{3.II} yields
\[
 \mI_{\bx_0,\w}\big(\lsm\bar x_1\\\bar  x_2\rsm,\bar u\big)\ge
 -\frac{1}{2}\,x_1^2\,n^+(\w)
\]
for any admissible pair. In fact, the right value is the infimum,
since it is reached at the admissible
pair $\big(\lsm \wit x_1\\0\rsm,\wit u\big)$ defined
from the solution $\lsm \wit x_1(t)\\0\\n^+(\wt)\,\wit x_1(t)\\0\rsm$ of \eqref{3.hamilex}$_\w$
with initial data $\lsm x_1\\0\\n^+(\w)\,x_1\\0\rsm$
via the feedback rule (analogous to \eqref{1.frulew}$_\w$)
\[
 \wit u(t)=
 [\,1\;0\,]\,\left[\,\!\begin{array}{c}n^+(\wt)\,\wit x_1(t)\\0\end{array}\!\right]
 -[\,1\;1\,]\,\left[\,\!\begin{array}{c}\wit x_1(t)\\0\end{array}\!\right]=
 (n^+(\wt)-1)\,\wit x_1(t)\,.
\]
This assertion follows from Lemma \ref{3.lemaV} (which does not require
Hypotheses \ref{3.hipos}) combined with \eqref{3.limex}.
\par
We finally observe that there is no way to know if the initial problem
of this example corresponds to a point $\w\in\W_0$ (although the probability is
1, as the measure of the set $\W_0$). In other words,
this procedure does not allow us to provide conditions under which the initial
minimization problem is solvable. This is one more sample of the
extreme complexity which may arise in the nonautonomous dynamics.
\end{exa}
\noindent{\bf Acknowledgement.} We thank an anonymous referee, whose
careful reading and comments have contributed to improve the presentation of the paper. 


\begin{thebibliography}{99}
\bibitem{cont} {\rm R. Conti},
        {\em Linear Differential Equations and Control},
        Academic Press, London, 1976.
\bibitem{copp1} {\rm W.A. Coppel},
        Dichotomies in Stability Theory,
        {\em Lecture Notes in Math.} {\bf 629},
        Springer-Verlag, Berlin, Heidelberg, New York, 1978.
\bibitem{fajn1} {\rm R.~Fabbri, R.~Johnson, C.~N\'{u}\~{n}ez},
        Rotation number for non-autonomous linear Hamiltonian systems I:
        Basic properties,
        {\em Z. Angew. Math. Phys.} {\bf 54} (2003), 484--502.
\bibitem{fajn4} {\rm R.~Fabbri, R.~Johnson, C.~N\'{u}\~{n}ez},
        On the Yakubovich Frequency Theorem for linear non-autonomous control processes,
        {\em Discrete Contin. Dynam. Systems, Ser.~A\/} {\bf 9} (3) (2003), 677--704.
\bibitem{fajn5} {\rm R.~Fabbri, R.~Johnson, C.~N\'{u}\~{n}ez},
        Disconjugacy and the rotation number for linear, non-autonomous Hamiltonian systems,
        {\em Ann. Mat. Pura App.} {\bf 185} (2006), S3--S21.
\bibitem{fink} {\rm A.M. Fink},
        Almost Periodic Differential Equations,
        {\em Lecture Notes in Math.} {\bf 377},
        Springer-Verlag, Berlin, Heidelberg, New York, 1974.
\bibitem{gijo} {\rm R. Giachetti, R. Johnson},
        The Floquet exponent for two-dimensional linear systems with
        bounded coefficients,
        {\em J. Math Pures et Appli.} {\bf 65} (1986), 93--117.
\bibitem{john12} {\rm R. Johnson},
        The recurrent Hill's equation,
        {\em J. Differential Equations\/} {\bf 46} (1982), 165--193.
\bibitem{jomo} {\rm R. Johnson, J. Moser},
        The rotation number for almost periodic potentials,
        {\em Comm. Math. Phys.} {\bf 84} (1982), 403--438.
\bibitem{jone0} {\rm R. Johnson, M. Nerurkar},
        On null controllability of linear systems with
        recurrent coefficients and constrained controls,
        {\em J. Dynam. Differential Equations\/} {\bf 4} (2) (1992), 259--273.
\bibitem{jnno} {\rm R.~Johnson, S.~Novo, C.~N\'{u}\~{n}ez, R.~Obaya},
        Uniform weak disconjugacy and principal solutions for linear Hamiltonian systems,
        {\em Recent Advances in Delay Differential and Difference Equations\/},
        Springer Proceedings in Mathematics \& Statistics {\bf 94} (2014), 131--159.
\bibitem{jonu2} {\rm R.~Johnson, C.~N\'{u}\~{n}ez},
        Remarks on linear-quadratic dissipative control systems
        {\em Discr. Cont. Dyn. Sys. B}, {\bf 20} (3) (2015), 889--914.
\bibitem{jnuo} {\rm R.~Johnson, C.~N\'{u}\~{n}ez, R.~Obaya},
        Dynamical methods for linear Hamiltonian systems with applications to
        control processes,
        {\em J. Dynam. Differential Equations\/} {\bf 25} (3) (2013), 679--713.
\bibitem{jonnf} {\rm R.~Johnson, R.~Obaya, S.~Novo, C.~N\'{u}\~{n}ez, R.~Fabbri},
        {\em Nonautonomous Linear Hamiltonian Systems:
        Oscillation, Spectral Theory and Control}
        Developments in Mathematics 36, Springer, 2016
\bibitem{jnot} {\rm \`{A}.~Jorba, C.~N\'{u}\~{n}ez, R.~Obaya, J.C.~Tatjer},
        Old and new results on Strange Nonchaotic Attractors,
        {\em Int. J. Bifurcation Chaos\/}, {\bf 17} (11) (2007), 3895--3928.
\bibitem{krat0} {\rm W.~Kratz},
        A limit theorem for monotone matrix functions,
        {\em  Linear Algebra Appl.} {\bf 194} (1993), 205--222.
\bibitem{nono} {\rm S.~Novo, C.~N\'{u}\~{n}ez, R.~Obaya},
        Ergodic properties and rotation number for linear Hamiltonian systems,
        {\em J. Differential Equations\/} {\bf 148} (1998), 148--185.
\bibitem{sase2} {\rm R.J. Sacker, G.R. Sell},
        Existence of dichotomies and invariant splittings for linear differential systems I,
        {\em J. Differential Equations\/} {\bf 15} (1974), 429--458.
\bibitem{sase5} {\rm R.J. Sacker, G.R. Sell},
        A spectral theory for linear differential systems,
        {\em J. Differential Equations\/}
        {\bf 27} (1978), 320--358.
\bibitem{selg} {\rm J. F. Selgrade},
        Isolated invariant sets for flows on vector bundles,
        {\em Trans. Amer. Math. Soc.} {\bf 203} (1975), 359--390.
\bibitem{sell2} {\rm G.R. Sell},
        {\em Topological Dynamics and Differential Equations\/},
        Van Nostrand Reinhold, London, 1971.
\bibitem{sept4}{\rm P.~\v{S}epitka},
        Riccati equations for linear Hamiltonian systems without
        controllability condition,
        {\em Discrete Contin. Dyn. Syst.} {\bf 39} (4) (2019), 1685-–1730.
\bibitem{sesi12}{\rm P.~\v{S}epitka, R.~\u{S}imon Hilscher},
        Principal and antiprincipal solutions at infinity of
        linear Hamiltonian systems,
        {\em J. Differential Equations\/} {\bf 259} (9) (2015), 4651–-4682.
\bibitem{sesi14}{\rm P.~\v{S}epitka, R.~\u{S}imon Hilscher},
        Genera of conjoined bases of linear Hamiltonian systems and
        limit characterization of principal solutions at infinity,
        {\em J. Differential Equations} {\bf 260} (6) (2016), 6581-–6603.
\bibitem{shyi4} {\rm W. Shen, Y. Yi},
        Almost Automorphic and Almost Periodic Dynamics in Skew-Product Semiflows,
        {\em Mem. Amer. Math. Soc.} {\bf 647},
        Amer. Math. Soc., Providence, 1998.
\bibitem{yaku3} {\rm V.A.~Yakubovich},
        Linear-quadratic optimization problem and the frequency theorem
        for periodic systems I,
        {\em Siberian Math. J.} {\bf 27} (4) (1986), 614--630.
\bibitem{yaku4} {\rm V.A.~Yakubovich},
        Linear-quadratic optimization problem and the frequency theorem
        for periodic systems II,
        {\em Siberian Math. J.} {\bf 31} (6) (1990), 1027--1039.
\end{thebibliography}
\end{document}